\documentclass[A4,10pt]{article}

\usepackage[top=3cm, bottom=2cm, left=2cm, right=2cm]{geometry}
\usepackage{amsmath,amsthm}
\usepackage{amssymb, amsmath}
\usepackage{amsfonts}
\usepackage{comment}
\usepackage{ascmac}
\usepackage{color}
\newtheorem{theo}{Theorem}[section]
\newtheorem{prop}{Proposition}[section]
\newtheorem{cor}{Corollary}[section]
\newtheorem{lem}{Lemma}[section]
\newtheorem{defi}{Definition}[section]
\newtheorem{rem}{Remark}[section]

\makeatletter

\@addtoreset{equation}{section}
\makeatother

\title{The existence and regularity theory for abstract semi-linear time-fractional evolution equations}
\author{Mizuki Kojima}
\date{}
\begin{document}
\maketitle
\begin{abstract}
	In this paper, we investigate abstract time-fractional evolution equations with nonlinear perturbations. We construct solutions of Lipschitz perturbation problems in arbitrary large time interval independent of the Lipschitz constants. We will extend well-known results for standard evolution equations such as the blow-up alternative, to the time-fractional evolution equations. We also prove the differentiability with respect to time of the solution when the perturbation is sufficiently smooth. The differentiability enables us to use the maximum principle. The theory on general Banach spaces enables us to deduce space regularity result easily.
\end{abstract}

\section{Introduction}

In this paper, we study the following abstract time-fractional Cauchy problem  
 \begin{equation}\label{eq. semilinear problem}
 	\left\{
 	\begin{aligned}
 	D_t^{\alpha}(u-u_0)(t) +A u(t) +B(t,u(t)) &=f(t)\ \mbox{in}\ (0,T],\\
 	u(0)&=u_0
 	\end{aligned}
 	\right.
 \end{equation}
in a Banach space $X$. The operator $D_t^{\alpha}$ is the Riemann-Liouville fractional differential operator of order $\alpha\in (0,1)$ defined by
\begin{equation}\label{eq. Riemanliouville derivative}
	D_t^{\alpha}(u)(t):=\frac{1}{\Gamma(1-\alpha)}\frac{d}{dt}\int_{0}^{t}(t-\tau)^{-\alpha}u(\tau)d\tau,
\end{equation}
and $A:D(A)\rightarrow X$ is a generator of some analytic semigroup.\\

The linear time-fractional evolution equation of the form
\begin{equation}\label{eq. linear cauchy problem}
\left\{
\begin{aligned}
D_t^{\alpha}(u-u_0)(t) +A u(t) &=f(t)\ \mbox{in}\ (0,T],\\
u(0)&=u_0
\end{aligned}
\right.
\end{equation}
was originally proposed in order to model the so-called anomalous diffusion phenomenon, which is different from the usual diffusion of materials based on Brownian motion (\cite{Masuda et al}, \cite{R. Metzler}).

There are many studies on the time-fractional Cauchy problem (\ref{eq. linear cauchy problem}). In \cite{MAH}, Mu, Ahmad and Huang investigated the existence and regularity of a solution when the external force term $f$ has the time regularity $f\in {\cal F}^{r,\beta}((0,T];X)$, where ${\cal F}^{r,\beta}((0,T];X)$ is a weighted H\"{o}lder continuous space. Luchko \cite{Luchko} established the maximum principle for (\ref{eq. linear cauchy problem}), but only under the condition that the solution is time differentiable. Guidetti \cite{Guidetti1},\cite{Guidetti2} investigated the maximal regularity properties of (\ref{eq. linear cauchy problem}) in an abstract framework. The mild-solution formula plays an important role in the theory of evolution equations, which is shown by \cite{Mahmoud}, \cite{perturbation} to be given by the following,
\begin{equation}\label{mild sol}
u(t)=\int_{0}^{\infty}h_{\alpha}(\theta)T(t^{\alpha}\theta) u_0 d\theta +\alpha\int_{0}^{t}\int_{0}^{\infty} (t-\tau)^{\alpha-1}\theta h_{\alpha}(\theta)T((t-\tau)^{\alpha}\theta)f(\tau)d\tau d\theta
\end{equation}
where $h_{\alpha}$ is a probability density function defined on $(0,\infty)$ such that
\[
\int_{0}^{\infty} \theta^{\gamma}h_{\alpha}(\theta) d\theta =\frac{\Gamma(1+\gamma)}{\Gamma(1+\alpha\gamma)}\ \ \mbox{forall}\ \gamma\in (-1,\infty),
\]
and $T(t)$ denotes the semigroup generated by $A$.

As for semi-linear and nonlinear problems, Zhou and Jiao \cite{perturbation} proved the existence of a local-in-time mild solution with nonlocal conditions,
\[
u(0)+g(u)=u_0
\]
where $g:C([0,T];X)\rightarrow X$ is a given function, while they imposed some restrictions on the growth of the perturbations in order to use the Banach fixed-point theorem. Akagi \cite{Akagi} showed the solvability of (\ref{eq. semilinear problem}) when $X$ is a Hilbert space and $A$ is replaced by subdifferential $\partial \varphi$ for some functional $\varphi$.

There are also some studies on specific models. Inc, Yusuf, Aliyu and Baleanu \cite{CAKG} investigated the one-dimensional time-fractional Allen-Cahn equation and Klein-Gordon equation. Liu, Cheng, Wangb and Zhao \cite{numerical} performed numerical analysis of time-fractional Allen-Cahn and Cahn-Hilliard phase-field models. These models describe the phase separation process of iron alloys. Zang, Sun \cite{fujita eq} concerned with the blow-up and global existence of a solution to the time-fractional Fujita equation. Zhang, Li and Su \cite{GL equation} investigated time-fractional Ginzburg-Landau equation which is used in the superconductivity theory.

It is also important for some approximation problems that the existence time $T>0$ can be chosen independently of the Lipschitz constant of the nonlinear term $B$. Regularized approximation equations enable us to consider a problem with some singularities. For example, in \cite{enzymatic}, the authors studied an equation that describes the heterogeneous chemical catalyst with the absorption term $B(u) = u^{-\beta}\chi_{u>0}$ by considering a regular approximation of $B$. In their discussion, it is essential that the solutions of approximation equations are defined in the same interval. Moreover, in the Stefan problem, the so-called penalty method is often used to construct the solution. This method also introduces the approximation equations to obtain a solution of the original problem. For the detail, we refer to \cite{KinderStam}, \cite{Heleshaw Stefan}. Nowadays there are some studies on Stefan problems in which the time-fractional heat equations is used as the governing equation, such as \cite{space fractional},\cite{timefractional stefan}.

\vspace{1ex}These examples motivate us to establish a unified and general theory
for the uniform-in-time solvability and regularity of solutions to (\ref{eq. semilinear problem}). The novelty of this paper is establishing the long-in-time existence of
a solution to the general semi-linear problem (\ref{eq. semilinear problem}), and a regularity of the solution in a unified approach. To this end, we introduce a weighted metric space of $X$-valued curves in such a way that the Banach fixed point theorem works merely with standard semigroup estimates. Moreover, we prove the regularity with respect to time variable $t$ by considering an auxiliary evolution equation that is obtained by formally differentiating (\ref{eq. semilinear problem}) with respect to $t$.
It is proved that a solution to the auxiliary equation turns out to be the derivative of the solution to (1.1) and this yields the desired regularity properties of the solution.

\vspace{1ex}Before stating the main theorems, we define the following weighted H\"{o}lder continuous function space.
\begin{defi}\label{weighted holder}
	For indices $0<\beta<r<1$, let ${\cal F}^{r,\beta}((0,T];X)$ denotes the function space of all continuous functions $f$ on $(0,T]$ with the following properties:
	\begin{enumerate}
		\item $\lim_{t\rightarrow 0}t^{1-r}f(t)$ exists.
		\item $f$ is $\beta$ H\"{o}lder continuous with the weight $s^{1-r+\beta}$, that is,
		\[
		\sup_{0\le s<t\le T}\frac{s^{1-r+\beta}\| f(t)-f(s) \|}{(t-s)^{\beta}}<\infty.
		\]
		\item the following holds,
		\[
		\lim_{t\rightarrow 0} w_f(t)=\lim_{t\rightarrow 0} \sup_{0\le s<t}\frac{s^{1-r+\beta}\| f(t)-f(s) \|}{(t-s)^{\beta}}=0.
		\]
	\end{enumerate}
\end{defi}
This function space become Banach space with the norm; 
\[
\|f\|_{{\cal F}^{r,\beta}}=\sup_{0\le t\le T}t^{1-r}\| f(t)\| +\sup_{0\le s<t\le T} \frac{s^{1-r+\beta}\| f(t)-f(s) \|}{(t-s)^{\beta}}.
\]
If there is no danger of confusion, we shall write ${\cal F}^{r,\beta}$ instead of ${\cal F}^{r,\beta}((0,T];X)$. It is obvious that the embedding ${\cal F}^{r,\beta'}\subset {\cal F}^{r,\beta}$ holds for $0<\beta<\beta'<r\le 1$ . When $g\in C^{\sigma}([0,T];X)$ satisfies $g(0)=0$, $f(t)=t^{r-1}g(t)$ belongs to ${\cal F}^{r,\sigma}((0,T];X)$. Also, for $0<\beta<r=1$, $f\in C^{\beta}$ belongs to ${\cal F}^{1,\beta}$. We refer to \cite{Yagi} in detail.

\vspace{1ex}First of all, we consider the Lipschitz perturbation problem where $B$ satisfies
\begin{equation}\label{eq. lipschitz condition}
\|B(t,u)-B(t,v)\|\le L \|u-v\|\ \ \mbox{forall}\ u,v\in X,\ t\ge 0
\end{equation}
for some Lipschitz constant $L>0$. See Appendix for the definition of fractional power operator $A^q$.
\begin{theo}\label{theo. main theorem1}
	Suppose that (\ref{eq. lipschitz condition}) is satisfied. Let $T, \beta, r$ satisfy $T>0$, $0<\beta<r\le 1$, $0<\max(\beta, 1-\alpha)<r$ and $\beta< \alpha$. Then, for all $f\in {\cal F}^{r,\beta}((0,T];X)$, $u_0 \in \overline{D(A)}$, the semi-linear problem (\ref{eq. semilinear problem}) has a unique solution 
	\[
	u\in C([0,T];X)\cap C((0,T];D(A))\cap {\cal F}^{r,\beta}((0,T];X).
	\]
	Furthermore, the time regularity of $u$ can be improved as follows,
	\begin{enumerate}
	\item In addition, if $u_0\in D(A^{q})$, $f\in C([0,T];X)$ for $\alpha^{-1}\beta \le q\le$ and $\beta<\alpha$, then $u\in C^{\beta}([0,T];X)$.
	\item In addition, if $u_0\in D(A)$ and $f\in {\cal F}^{1,\beta}((0,T];X)$, then $u\in C^{\alpha}([0,T];X)$.
	\end{enumerate}
\end{theo}

We briefly mention the time regularity. At first, the additional assumption of $u_0\in  D(A^q)$ deduces the H\"{o}lder continuity of the term
\[
t\mapsto \int_{0}^{\infty}h_{\alpha}(\theta)T(t^{\alpha}\theta) u_0 d\theta
\]
in (\ref{mild sol}). In general, we can only expect the strong continuity of $T(t)$ at $t=0$. Therefore $u_0$ must have sufficient regularity. Also, roughly speaking, the assumption of $u_0\in D(A)$ and $f\in {\cal F}^{1,\beta}((0,T];X)$ means that $D_t^{\alpha}(u-u_0)$ is in $C([0,T];X)$. That is formally, the $\alpha$-order derivative is continuous on $[0,T]$, hence the original function must be in $C([0,T];X)$. We justify this argument in section\ \ref{section. Fractional power of operator}.

\vspace{1ex}Differentiability of the solution is essential for the use of maximum principle. However, to the author's best knowledge there has been no research on the differentiability in time of the solution in semi-linear problem (\ref{eq. semilinear problem}). We deduce the regularity in time of the solution when $f$ and $B$ are sufficiently regular to use the maximum principle shown by \cite{Luchko} which is useful in the study of the quantitative properties of the solution. We do not assume time dependence of $B$ in this theorem. The condition for indices $r', \beta'$ is required to obtain the well-posedness of the auxiliary equation in the proof.

\begin{theo}\label{theo. main theorem2}
	We assume the following conditions,
	\begin{enumerate}
		\item Time regularity
		\[
		f\in {\cal F}^{r,\beta}((0,T];X)\cap C^{1}((0,T];X),\ \ f'\in {\cal F}^{r',\beta'}((0,T];X)
		\]
		where we take $r, \beta$ as in Theorem\ \ref{theo. main theorem1} and $r',\beta'$ such that $0<\beta'<r'\le 1$, $0<\max(\beta', 1-\alpha)<r'$, $\beta'<\alpha$ and $\alpha/2\le r'$.
		\item The nonlinear part $B:X\rightarrow X$ is Fr\'{e}chet differentiable, its derivative $B':X\rightarrow B(X)$ is uniformly  bounded, that is, there exists $L>0$ such that $\|B(u)\|_{B(X)}\le L$ for all $u\in X$, and $B'$ is also Lipschitz continuous with Lipschitz constant $L'>0$.
		\item The initial condition $Au_0+B(u_0)-f(0)\in \overline{D(A)}$ are satisfied.
	\end{enumerate}
Then, the solution of the problem (\ref{eq. semilinear problem}) is differentiable with respect to time in $(0,T]$ and $u'\in {\cal F}^{\alpha-\gamma,\gamma}((0,T];X)$ where
\[
0<\gamma<\alpha/2,\ \alpha-\gamma\le r',\ \gamma\le \beta'.
\]
\end{theo}

Also, we organize the the general result for blow-up of the solution with locally Lipschitz perturbation in the sense that

\begin{equation}\label{eq. local lipschitz}
\|B(t,u)-B(t,v)\|\le L(\rho)\|u-v\|\ \ \mbox{for all}\ u,v\in \{w\in X; \|w\|\le \rho\}
\end{equation}
where $L:\mathbb{R}\rightarrow \mathbb{R}$ is increasing function, and for each $u\in X$, $t\mapsto B(t,u)$ is bounded, with reference to the idea of \cite{fujita eq}, \cite{GL equation}.

\begin{theo}\label{theo. main theorem3}
	Suppose that $B$ satisfies (\ref{eq. local lipschitz}). Then, (\ref{eq. semilinear problem}) has the unique local solution. Furthermore, suppose that $T^*>0$ is the maximal existence time. Then, the solution satisfies either one of the following:
	\begin{enumerate}
		\item $T^*=\infty$.
		\item $T^*<\infty$ and $\displaystyle \lim_{t\rightarrow T^*}\|u(t)\|=\infty$.
	\end{enumerate}
\end{theo}

The outline of this paper is as follows. In section\ \ref{section. preliminaries}, we introduce some notations, function spaces, and some preliminary results which shall be used later. In section\ \ref{section. Fractional power of operator}, we introduce the abstract theory of fractional power of operator, and characterize the time-fractional operator (\ref{eq. Riemanliouville derivative}) by the abstract framework. In section\ \ref{section. main theorem1}, we prove Theorem\ \ref{theo. main theorem1}. The outline of the proof is as follows. At first, we show that there exists a unique mild solution no matter how large the $T$ and $L$ are. The mild solution essentially becomes the usual one for (\ref{eq. semilinear problem}) because of the existence theorem given by \cite{MAH}. In section\ \ref{section. main theorem2 time diff}, we prove the differentiability in time of the solution when $f$ and $B$ have appropriate regularities. This result is essential to use maximum principle due to \cite{Luchko}. In section\ \ref{section. main theorem3 locally lipschitz perturbation}, we investigate the locally Lipschitz perturbation problems. In the usual evolution equation theory, the property so-called blow-up alternative holds for the equations equipped with locally Lipschitz perturbations. We verify the same result as usual must holds. The key idea is the extension of solutions via integral equations due to \cite{fujita eq}, \cite{GL equation}. In section\ \ref{section. application}, we apply our results to a specific model. We investigate the combustion model with the perturbation part $B(u)=e^{-1/u}$. This is treated by Ikeda, Mimura \cite{combustion} to describe the reaction-diffusion processes with the generation of heats, so the time-fractional diffusion model must be important to consider some complicated physical phenomena. The perturbation item comes from the Arrhenius law to states the dependence of reaction rates on absolute temperature $u$.

\section{preliminaries}\label{section. preliminaries}

\subsection{Notations}

Let $X$ be a Banach space equipped with the norm $\|\cdot\|$. The liner operator $A:X\supset D(A)\rightarrow X$ is called sectorial, if it satisfies the following conditions,
\begin{equation}\label{sectrial1}
\sigma(A)\subset \Sigma_{\omega}=\{  \lambda\in \mathbb{C};\ |\arg \lambda |< \omega\}\ \mbox{for some}\ 0<\omega< \pi/2,
\end{equation}
\begin{equation}\label{sectrial2}
\| (\lambda-A)^{-1}\|_{B(X)}\le \frac{M}{|\lambda|}\ \ \forall \lambda\notin\Sigma_{\omega}\ \mbox{for some}\ M>0
\end{equation}
where $\sigma(A)$ is the spectrum set of $A$. The condition (\ref{sectrial1}) implies that the resolvent set of $A$ contains $0$, that is, $A^{-1}\in B(X)$. More precisely, for $\delta<\|A^{-1}\|^{-1}$, $\{ \lambda\in \mathbb{C}; |\lambda|\le \delta \}\subset \rho(A)$. We do not assume the density of $D(A)$ in order to enable us to apply the theory to a wide variety of function spaces to gain some space regularity. We can define the fractional power of the sectorial operator, which we shall organize in Appendix.

As we mentioned in the introduction, according to \cite{Mahmoud}, \cite{perturbation},  the solution of (\ref{eq. linear cauchy problem}) necessarily satisfies the following representation formula, which is the main tool to deduce some properties of the solution.
\[
u(t)=\int_{0}^{\infty}h_{\alpha}(\theta)T(t^{\alpha}\theta) u_0 d\theta +\alpha\int_{0}^{t}\int_{0}^{\infty} (t-\tau)^{\alpha-1}\theta h_{\alpha}(\theta)T((t-\tau)^{\alpha}\theta)f(\tau)d\tau d\theta
\]
where $h_{\alpha}$ is a probability density function defined on $(0,\infty)$ such that
\[
\int_{0}^{\infty} \theta^{\gamma}h_{\alpha}(\theta) d\theta =\frac{\Gamma(1+\gamma)}{\Gamma(1+\alpha\gamma)}\ \ \mbox{forall}\ \gamma\in (-1,\infty),
\]
and $T(t)$ denotes the semigroup generated by $A$ which satisfies the following estimates;
\begin{equation}\label{evaluate for semigroup1}
\begin{aligned}
&\| T(t)\|_{B(X)}\le C_1e^{-C_2 t}\le C_1\ \ \forall t\ge 0,\\
&\| A^{\delta}T(t)\|_{B(X)}\le C_{\delta}t^{-\delta}\ \ \forall t>0,\ \forall \delta>0.
\end{aligned}
\end{equation}
Moreover, for all $0<\delta\le 1$,
\begin{equation}\label{evaluate for semigroup2}
\| [T(t)-I]A^{-\delta} \|\le \left\| \int_{0}^{t} A^{1-\delta}T(\tau)d\tau \right\|\le C_{\delta}\int_{0}^{t} \tau^{\delta-1}d\tau\le C_{\delta}t^{\delta},\ \ 0<t<\infty.
\end{equation}

\subsection{Results on the existence and regularity}
We shall employ the following result when we conclude the mild solution of (\ref{eq. semilinear problem}) necessarily be the usual one.
\begin{prop}\label{existence}
	(see \cite{MAH}) Suppose that $f\in{\cal F}^{r,\beta}((0,T];X)$ and $u_0\in \overline{D(A)}$. Then, for all $0<\max(\beta,1-\alpha)<r\le1$ and $0<\beta<\alpha<1$, (\ref{eq. linear cauchy problem}) has unique solution $u \in C([0,T];X)\cap C((0,T];D(A))$.
\end{prop}
It is known that the following maximal regularity results for (\ref{eq. linear cauchy problem}) hold.

\begin{prop}\label{space regularity}
	(see \cite{CGL}, \cite{Guidetti2}) For $\theta \in (0,1)$, the following conditions are necessary and sufficient, in order that there exists a unique solution $u$ of (\ref{eq. linear cauchy problem}) such that $D_{t}^{\alpha}(u-u_0), Au$ are bounded with values in $(X,D(A))_{\theta,\infty}$;
	\begin{enumerate}
		\item $f\in C([0,T];X)\cap B([0,T];(X,D(A))_{\theta,\infty})$,
		\item $u_0\in D(A), Au_0\in (X,D(A))_{\theta,\infty}$.
	\end{enumerate}
\end{prop}

\begin{prop}\label{time regularity}
	(see \cite{CGL}, \cite{Guidetti2}) For $\theta \in (0,\alpha)$, the following conditions are necessary and sufficient, in order that there exists a unique solution $u$ of (\ref{eq. linear cauchy problem}) such that $D_{t}^{\alpha}(u-u_0), Au$ belong to $C^{\theta}([0,T];X)$;
	\begin{enumerate}
		\item $f\in C^{\theta}([0,T];X)$,
		\item $u_0\in D(A), Au_0+f(0)\in (X,D(A))_{\theta/\alpha,\infty}$.
	\end{enumerate}
\end{prop}

When we observe the quantitative property, the next statement is essential.
\begin{prop}\label{max}
	(see \cite{Luchko}, \cite{space fractional}) Let $g\in C([0,T])\cap C^{1}((0,T])$ satisfies $g'\in L^1(0,T)$ and attains its maximum over the interval $[0,T]$ at $t_0\in (0,T]$, then, for all $0<\alpha<1$, 
	\[
	\int_{0}^{t_0}(t_0-\tau)^{-\alpha}g'(\tau)d\tau \ge 0.
	\]
	Moreover, if $g$ is not constant on $[0,t_0]$, then
	\[
	\int_{0}^{t_0}(t_0-\tau)^{-\alpha}g'(\tau)d\tau > 0.
	\]
\end{prop}

We note that if $g$ is differentiable, then
\[
\begin{aligned}
D_t^{\alpha}(g-g(0))(t)&=\frac{1}{\Gamma(1-\alpha)}\frac{d}{dt}\int_{0}^{t}(t-\tau)^{-\alpha}(g(\tau)-g(0))d\tau\\
&=\frac{1}{\Gamma(1-\alpha)}\int_{0}^{t}(t-\tau)^{-\alpha} g'(\tau)d\tau
\end{aligned}
\]
where the last term is the so-called Caputo derivative. Therefore, we have to confirm the time differentiability of the solution when we use Proposition\ \ref{max} to (\ref{eq. semilinear problem}), or just (\ref{eq. linear cauchy problem}).

\section{Fractional power of operator and time-fractional derivative}\label{section. Fractional power of operator}

In this section, we introduce a formalization of the time-fractional derivative in an abstract framework due to \cite{Guidetti1},\cite{Guidetti2}. Let an operator $B$ be as follows;
\begin{equation}\label{def of B}
\left\{
\begin{aligned}
&B: \{v\in C^1([0,T];X): v(0)=0\}\rightarrow C([0,T];X),\\
&Bv =-\frac{dv}{dt}.
\end{aligned}
\right.
\end{equation}
For $0<\alpha<1$, the power of the operator $B$ is as follow;
\begin{equation}\label{def of Balpha}
B^{-\alpha}=\frac{1}{2\pi i}\int_{\Gamma} \lambda^{-\alpha}(\lambda-B)^{-1}d\lambda.
\end{equation}
where we take the path $\Gamma$ as in the Appendix. We can confirm that the formula coincides with,
\begin{equation}\label{charactor of B-alpha}
(B^{-\alpha}f)(t) =\frac{1}{\Gamma(\alpha)}\int_{0}^{t} (t-\tau)^{\alpha-1}f(\tau)d\tau.
\end{equation}

In the  abstract point of view, the inverse of (\ref{def of Balpha}) $B^{\alpha}:=(B^{-\alpha})^{-1}$ is the $\alpha$ power of $B$. And the domain of $B^{\alpha}$ is the range of $B^{-\alpha}$, that is, $D(B^{\alpha})=R(B^{-\alpha})$. On the other hand, we can check the formula (\ref{charactor of B-alpha}) is inverse mapping of $g\mapsto D_t^{\alpha}(g-g(0))$. This implies that $D_t^{\alpha}$ and $B^{\alpha}$ is the same operator. These properties are mentioned in \cite{Guidetti1},\cite{Guidetti2}.

\begin{prop}
	The operator $B$ defined (\ref{def of B}) is sectrial in $C([0,T];X)$, and its fractional power is given by (\ref{charactor of B-alpha}). In particular,
	\[
	B^{\alpha}(g-g(0))=D_t^{\alpha}(g-g(0)).
	\]
\end{prop}

\begin{proof}
	For all $\lambda \in \mathbb{C}$ and $f\in C([0,T];X)$,
	\[
	\begin{aligned}
	v = (\lambda -B)^{-1}f&\Leftrightarrow \frac{dv}{dt}+\lambda v = f,\ v(0)=0\\
	&\Leftrightarrow v(t)=\int_{0}^{t} e^{-\lambda(t-s)}f(s)ds
	\end{aligned}
	\]
	Therefore, the power of $B$ must be calculated as follow;
	\[
	\begin{aligned}
	(B^{-\alpha}f)(t) &=\frac{1}{2\pi i} \int_{\Gamma} \lambda^{-\alpha} \left[ \int_{0}^{t}e^{-\lambda(t-s)}f(s)ds \right]d\lambda\\
	&=\frac{1}{2\pi i} \int_{0}^{t} f(s) \left[ \int_{\Gamma} \lambda^{-\alpha} e^{-\lambda (t-s)}d\lambda \right] ds.
	\end{aligned}
	\]
	So, It is sufficient to show that for all $b> 0$,
	\[
	\frac{1}{2\pi i}\int_{\Gamma} \lambda^{-\alpha} e^{-b\lambda}d\lambda = \frac{1}{\Gamma(\alpha)} b^{\alpha-1}
	\]
	holds. This formula is obtained by the Hankel's formula and the transformation of the path $\Gamma$ to the Hankel contour. The resolvent estimate (\ref{sectrial2}) shall be obtained easily;
	\[
	\begin{aligned}
	\|v(t)\| &\le \int_{0}^{t}|e^{-\lambda(t-s)}|\|f(s)\|ds\\
	&\le \frac{1}{|\lambda|} \left[ 1- e^{-|\lambda|t} \right]\|f\|_{C([0,T];X)}\le \frac{1}{|\lambda|} \|f\|_{C([0,T];X)}.
	\end{aligned}
	\]
\end{proof}

\begin{prop}
	(see \cite{Guidetti1},\cite{Guidetti2}) Let $B$ be as in (\ref{def of B}). Then,
	\begin{equation}
	(C([0,T];X),\ D(B))_{\theta,\infty}=\{ f\in C^{\theta}([0,T];X): f(0)=0 \}
	\end{equation}
	satisfies with the equivalent norms. 
\end{prop}
Proposition \ref{inclusion of domain power of op} deduces the inclusion for $\beta >\theta$,
\begin{equation}\label{eq. inclusion of D(Balpha)}
\{ f\in C^{\beta}([0,T];X);\ f(0)=0\} \subset D(B^{\theta})\subset \{ f\in C^{\theta}([0,T];X);\ f(0)=0\}.
\end{equation}

In symbolic terms, a H\"{o}lder continuous function even slightly stronger than $\alpha$ is necessarily $\alpha$-order differentiable, and its derivative is continuous on $[0,T]$. Also, if a $\alpha$-order derivative is continuous on $[0,T]$, then the original function is necessarily $\alpha$-order H\"{o}lder continuous. This is an analogy with the integro-differential.


\begin{rem}
	Let $Y=L^p(0,T),D(B)=W^{1,p}(0,T)$, then we get
	\[
	D(B^{\alpha})\subset (L^{p}(0,T), W^{1,p}(0,T))_{\alpha,\infty}=B^{\alpha,p,\infty}(0,T)
	\]
	where we define Besov space as follow;
	\[
	B^{s,p,q}:=(L^p,W^{m,p})_{s/m,q}.
	\]
	For any bounded $\Omega\subset \mathbb{R}^N$ with sufficiently smooth boundary, we employ the following inclusion  (see, \cite{Adams}),
	\[
	sp>N\Rightarrow B^{s,p,q}(\Omega)\subset C(\overline{\Omega}).
	\]
	When $N=1, m=1$, this embedding holds if $\alpha>1/p$. Therefore, in the  Lebesgue and Sobolev space perspective, $\alpha>1/p$ diffirentiable function naturally become continuous. This fact has already been mentioned in \cite{Akagi} when $p=2$ and $X$ is Hilbert space.
\end{rem}

\section{Proof of Theorem \ref{theo. main theorem1}}\label{section. main theorem1}
 
We first show the existence of time global mild-solution in the framework of continuous functions. Then, the regularity guarantees that the mild-solution is actually a classical one. We solve the semi-linear problem (\ref{eq. semilinear problem}) by using the Banach-fixed point theorem with reference to \cite{Akagi}. The idea is to define a function space with a small enough norm for given $T$ and $L$. We equip the norm
\[
\|u\|_{\mu} = \sup_{0\le t\le T}e^{-\mu t} \|u(t)\|
\]
to $C([0,T];X)$. We define the map $S:C([0,T];X)\rightarrow C([0,T];X)$ by
\[
\begin{aligned}
S:v\mapsto u(t)&=\int_{0}^{\infty}h_{\alpha}(\theta)T(t^{\alpha}\theta)d\theta u_0\\
&\hspace{3cm}+\alpha \int_{0}^{t}\int_{0}^{\infty} (t-\tau)^{\alpha-1}\theta h_{\alpha}(\theta)T((t-\tau)^{\alpha-1}\theta)\left[f(\tau) -B(v(\tau)) \right]d\tau d\theta.
\end{aligned}
\]

\begin{prop}\label{prop. fixed point theorem for S}
	For all $L>0,\ T>0$, there exists sufficiently large $\mu >0$ such that the $S$ becomes contraction mapping.
\end{prop}

\begin{proof}
	
	For arbitrary $v_1,v_2\in C([0,T];X)$, let $u_1=Sv_1, u_2=Sv_2$. We can easily see that
	\[
	\begin{aligned}
	\| u_1(t)-u_2(t)\|&\le  \frac{LC_1\alpha}{\Gamma(1+\alpha)} \int_{0}^{t}  e^{\mu t} e^{-\mu(t-\tau)}e^{-\mu \tau} (t-\tau)^{\alpha-1}\| v_1(\tau)-v_2(\tau)\|d\tau\\
	e^{-\mu t}\| u_1(t)-u_2(t)\|&\le \frac{LC_1\alpha}{\Gamma(1+\alpha)} \int_{0}^{t}  e^{-\mu(t-\tau)} (t-\tau)^{\alpha-1}d\tau\| v_1-v_2\|_{\mu} .
	\end{aligned}
	\]
	We note that 
	\[
	\int_{0}^{t}  e^{-\mu(t-\tau)} (t-\tau)^{\alpha-1}d\tau\le \int_{0}^T e^{-\mu \tau}\tau^{\alpha-1}d\tau
	\]
	and
	\[
	e^{-\mu \tau}\tau^{\alpha-1}\le \tau^{\alpha-1},\ \ e^{-\mu \tau}\tau^{\alpha-1}\rightarrow 0\ \mbox{pointwize in}\ (0,T)\ \mbox{as}\ \mu\rightarrow\infty
	\]
	So, we set
	\[
	A_\mu = \int_{0}^T e^{-\mu \tau}\tau^{\alpha-1}d\tau \rightarrow 0\ \mbox{as}\ \mu\rightarrow \infty
	\]
	to conclude
	\[
	\|u_1-u_2\|_{\mu}=\sup_{0\le t\le T}e^{-\mu t}\| u_1(t)-u_2(t)\|\le A_{\mu}\frac{LC_1\alpha}{\Gamma(1+\alpha)} \|v_1-v_2\|_{\mu}.
	\]
	In particular, we choose so large $\mu>0$ that $A_{\mu}LC_1\alpha/\Gamma(1+\alpha)<1$.
\end{proof}

\begin{prop}\label{prop. time regularity of S_0}
	Let
	\[
	S_0(t)=\int_{0}^{\infty}h_{\alpha}(\theta)T(t^{\alpha}\theta) d\theta u_0
	\]
	for $u_0\in X$.
	\begin{enumerate}
		\item If $u_0\in \overline{D(A)}$, then $S_0\in{\cal F}^{r,\beta}((0,T];X))$ for all $0<\beta<r<1$ and $\beta<\alpha$. 
		\item If $u_0\in D(A^{q})$ for $\alpha^{-1}\beta\le q\le 1$ and $\beta<\alpha$, then $S_0\in C^{\beta}([0,T];X)$.
	\end{enumerate}
\end{prop}
\begin{proof}
	\begin{enumerate}
		\item The strong continuity of the semigroup at $t=0$ guarantees the existence of $\lim_{t\rightarrow 0} t^{1-r}\|S_0(t)\|$. Suppose that $s<t$, then for arbitrary $x\in X$ and $0<\eta<1$,
		\[
		\begin{aligned}
		\| T(t)x- T(s)x\| &\le \int_{s}^{t} \|AT(\tau)\| \|x\|d\tau\\
		&= \int_{s}^{t} \|T(\tau-s)A^{1-\eta}A^{\eta}T(s)\|d\tau \|x\|\\
		&\le C_{\eta} \int_{s}^{t} (\tau-s)^{\eta-1} s^{-\eta} \|x\|\\
		&\le C_{\eta} (t-s)^{\eta} s^{-\eta}\|x\|
		\end{aligned}
		\]
		Therefore,
		\[
		\begin{aligned}
		\|S_0(t)-S_0(s)\| &\le \int_{0}^{\infty} h_{\alpha}(\theta) \|T(t^{\alpha}\theta)u_0-T(s^{\alpha}\theta)u_0\|d\theta\\
		&\le C_{\eta}\int_{0}^{\infty} h_{\alpha}(\theta)  (t^{\alpha}\theta-s^{\alpha}\theta)^{\eta} s^{-\alpha\eta}\theta^{-\eta}d\theta \|u_0\|\\
		&\le C_{\eta,\alpha}\|u_0\| (t-s)^{\alpha\eta} s^{-\alpha\eta}
		\end{aligned}
		\]
		Hence we take $\eta=\alpha^{-1}\beta$ to evaluate
		\[
		\frac{s^{1-r+\beta}\|S_0(t)-S_0(s)\|}{(t-s)^{\beta}}\le C_{\eta,\alpha} \|u_0\| s^{1-r}\rightarrow 0
		\]
		as $t\rightarrow 0$.
		
		\item The condition $u_0\in D(A^{q})$ deduces
		\[
		\begin{aligned}
		\|S_0(t)-u_0\|&= \int_{0}^{\infty} h_{\alpha}(\theta) \| (T(t^{\alpha}\theta)-I) A^{-q}\| \|A^{q} u_0\|d\theta\\
		&\le C_q \int_{0}^{\infty}h_{\alpha}(\theta)\theta^{q} t^{\alpha q}\| A^q u_0\|d\theta\\
		&\le \frac{C_q\Gamma(1+q)\|A^{q}u_0\|}{\Gamma(1+\alpha q)}t^{\beta}.
		\end{aligned}
		\]
		Therefore, we prove the H\"{o}lder continuity at $t=0$. The continuity on $(0,T]$ can be obtained by using the following estimate,
		\[
		\|T(t^{\alpha}\theta)u_0 -T(s^{\alpha}\theta)u_0\|\le \|(T(t^{\alpha}\theta - s^{\alpha}\theta)- I )A^{-q}\| \|T(s^{\alpha}\theta)\| \|A^{q}u_0\|
		\]

	\end{enumerate}
\end{proof}

\begin{prop}\label{prop. time regularity of S_1}
	Let
	\[
	S_1(g)(t):=\int_{0}^{t}\int_{0}^{\infty} (t-\tau)^{\alpha-1}\theta h_{\alpha}(\theta)T((t-\tau)^{\alpha-1}\theta)g(\tau)d\tau d\theta.
	\]
	\begin{enumerate}
		\item  If $g\in {\cal F}^{r,\beta}((0,T];X)$ for $0<\beta<r\le 1$ and $\beta<\alpha$, then $S_1(g)\in {\cal F}^{r,\beta}((0,T];X)$
		\item  If $g\in C([0,T];X)$, then $S_1(g)\in C^{\beta}([0,T];X)$.
	\end{enumerate}
	
\end{prop}
\begin{proof}
Let $g\in {\cal F}^{r,\beta}((0,T];X)$. The simple calculation yields
\[
\begin{aligned}
&\| S_1(g)(t)-S_1(g)(s)\|\le \left\| \int_{0}^{t}(t\text{'s formula}) \pm\int_{0}^{s}(t\text{'s formula}) -\int_{0}^{s}(s\text{'s formula})\right\|\\
&\le \int_{s}^{t} \int_{0}^{\infty} \theta h_{\alpha}(\theta) (t-\tau)^{\alpha-1} \| T((t-\tau)^{\alpha}\theta)\|_{B(X)} \tau^{r-1}\tau^{1-r}\|g(\tau)\|d\tau d\theta\\
&\hspace{1cm}+\int_{0}^{s}\int_{0}^{\infty}\theta h_{\alpha}(\theta)\| \left((s-\tau)^{\alpha-1} T((s-\tau)^{\alpha}\theta)- (t-\tau)^{\alpha-1}T((t-\tau)^{\alpha}\theta)\right) g(\tau)\|d\tau d\theta\\
&\le \frac{C_1\|g\|_{{\cal F}^{r,\beta}}}{\Gamma(1+\alpha)} \int_{s}^{t} (t-\tau)^{\alpha-1} \tau^{r-1}d\tau\\
&\hspace{1cm}+\|g\|_{{\cal F}^{r,\beta}}\int_{0}^{s} \theta h_{\alpha}(\theta)(s-\tau)^{\alpha-1}\|T((t-\tau)^{\alpha}\theta)-T((s-\tau)^{\alpha}\theta)\|\tau^{r-1}d\tau d\theta\\
&\hspace{1cm}+\frac{C_1}{\Gamma(1+\alpha)} \int_{0}^{s}  \left| (s-\tau)^{\alpha-1}- (t-\tau)^{\alpha-1} \right| \|g(\tau)\|d\tau \\
&=: (J_1+J_2+J_3)
\end{aligned}
\]

The estimate for $J_1, J_2, J_3$ is as follows,

\[
\begin{aligned}
J_1 &=\frac{C_1\|g\|_{{\cal F}^{r,\beta}}}{\Gamma(1+\alpha)}\int_{s}^{t}(t-\tau)^{\alpha-1}\tau^{\beta}\tau^{r-\beta-1}d\tau\le \frac{C_1\|g\|_{{\cal F}^{r,\beta}}}{\Gamma(1+\alpha)} \int_s^t (t-\tau)^{\alpha-1}d\tau\ t^{\beta}s^{r-\beta-1}\\
s^{1-r+\beta}J_1 &\le \frac{C_1\|g\|_{{\cal F}^{r,\beta}}}{\Gamma(1+\alpha)}t^{\beta} \int_s^t(t-\tau)^{\alpha-1}d\tau=\frac{C_1\|g\|_{{\cal F}^{r,\beta}}}{\alpha\Gamma(1+\alpha)}t^{\beta} (t-s)^{\alpha}.
\end{aligned}
\]

We select arbitrary $0<\delta<1$ to gain the following estimate,
\[
\begin{aligned}
&(s-\tau)^{\alpha-1}\|T((t-\tau)^{\alpha}\theta)-T((s-\tau)^{\alpha}\theta)\|\tau^{r-1}\\
&=(s-\tau)^{\alpha-1}\tau^{r-1} \|\left[ T((t-\tau)^{\alpha}\theta-(s-\tau)^{\alpha}\theta)-I\right]A^{-\delta}A^{\delta}T((s-\tau)^{\alpha}\theta) \|\\
&\le C_{\delta} (s-\tau)^{\alpha-1}\tau^{r-1}\left[ (t-\tau)^{\alpha}\theta-(s-\tau)^{\alpha}\theta \right]^{\delta} (s-\tau)^{-\delta\alpha}\theta^{-\delta}\\
&= C_{\delta}(s-\tau)^{\alpha(1-\delta)-1}\tau^{r-1}(t-s)^{\alpha\delta}.
\end{aligned}
\]
Therefore
\[
\begin{aligned}
J_2&\le C_{\delta}\int_{0}^s (s-\tau)^{\alpha(1-\delta)-1}\tau^{r-1}(t-s)^{\alpha\delta}d\tau\\
&\le C_{\delta} B(\alpha(1-\delta),r) s^{\alpha(1-\delta) + \beta}s^{r-\beta-1} (t-s)^{\alpha\delta}\\
s^{1-r+\beta}J_2&\le C_{\delta}B(\alpha(1-\delta),r) s^{\alpha(1-\delta)+\beta} )(t-s)^{\alpha\delta}
\end{aligned}
\]

We note that $\tau\mapsto \tau^{\alpha-1}$ is in ${\cal F}^{\alpha-\sigma, \sigma}$ for all $\sigma <\alpha$. Then, there exists some $C_{\sigma}$ such that
\[
\begin{aligned}
J_3&= \frac{C_1}{\Gamma(1+\alpha)} \int_{0}^{s} \left[ (s-\tau)^{\alpha-1}- (t-\tau)^{\alpha-1} \right] \|g(\tau)- g(s)\|d\tau\\
&\hspace{2cm}+ \frac{C_1}{\Gamma(1+\alpha)} \int_{0}^{s} \left[ (s-\tau)^{\alpha-1} - (t-\tau)^{\alpha-1} \right] \|g(s)\|d\tau\\
&\le \frac{C_1C_{\sigma}\|g\|_{{\cal F}^{r,\beta}}}{\Gamma(1+\alpha)}(t-s)^{\sigma} \int_{0}^{s} (s-\tau)^{\alpha-1-2\sigma} (s-\tau)^{\beta} \tau^{r-1-\beta}d\tau\\
&\hspace{2cm} + \frac{C_1\|g\|_{{\cal F}^{r,\beta}}}{\Gamma(1+(\alpha))} s^{r-1} \int_{0}^{s} \left[ (s-\tau)^{\alpha-1} - (t-\tau)^{\alpha-1} \right]d\tau
\end{aligned}
\]
We denote the appropriate constants as $C$ to deduce
\[
\begin{aligned}
J_3&\le C (t-s)^{\sigma} B(\alpha+\beta-2\sigma, r-\beta) s^{\alpha-2\sigma+\beta} s^{r-1-\beta} + C s^{r-1-\beta}s^{\beta}(t-s)^{\alpha}\\
s^{1-r+\beta}J_3&\le C  s^{\alpha-2\sigma+\beta} (t-s)^{\sigma}+ Cs^{\beta}(t-s)^{\alpha}.
\end{aligned}
\]
Especially, we select $\sigma =\beta$.

Let $g\in C([0,T];X)$. For $s<t$, a simple calculation yields
	\[
	\begin{aligned}
	&\| S_1(g)(t)-S_1(g)(s)\|\le \left\| \int_{0}^{t}(t\text{'s formula}) \pm\int_{0}^{s}(t\text{'s formula}) -\int_{0}^{s}(s\text{'s formula})\right\|\\
	&\le \int_{s}^{t} \int_{0}^{\infty} \theta h_{\alpha}(\theta) (t-\tau)^{\alpha-1} \| T((t-\tau)^{\alpha}\theta)\|_{B(X)}\|g(\tau)\|d\tau d\theta\\
	&\hspace{8ex}+\int_{0}^{s}\int_{0}^{\infty} \theta h_{\alpha}(\theta) \| (s-\tau)^{\alpha-1} T((s-\tau)^{\alpha}\theta) -(t-\tau)^{\alpha-1} T((t-\tau)^{\alpha}\theta) \|g(\tau)\|d\tau d\theta\\
	&\le \frac{C_1\|g\|}{\Gamma(1+\alpha)} \int_{s}^{t} (t-\tau)^{\alpha-1}d\tau+ \frac{C_1\|g\|}{\Gamma(1+\alpha)}\int_{0}^{s} |(s-\tau)^{\alpha-1}-(t-\tau)^{\alpha-1}|d\tau\\
	&\hspace{8ex}+ \|g\|\int_{0}^{s}\int_{0}^{\infty}\theta h_{\alpha}(\theta) \| T((t-\tau)^{\alpha}\theta)-T((s-\tau)^{\alpha}\theta)\| (s-\tau)^{\alpha-1}d\tau d\theta\\
	&=(J_1+J_2+J_3)
	\end{aligned}
	\]
	Now, the evalutation of $J_1, J_2$ is as follows,
	\[
	\begin{aligned}
	J_1&=\frac{C_1\|g\|}{\Gamma(1+\alpha)}\int_{s}^{t} (t-\tau)^{\alpha-1}d\tau\le \frac{C_1\|g\|}{\alpha\Gamma(1+\alpha)}(t-s)^{\alpha}\\
	J_2&= \frac{C_1\|g\|}{\Gamma(1+\alpha)}\int_{0}^{s} |(s-\tau)^{\alpha-1}-(t-\tau)^{\alpha-1}| d\tau \\
	&= \frac{C_1\|g\|}{\alpha\Gamma(1+\alpha)}\left( s^{\alpha} - t^{\alpha} +(t-s)^{\alpha} \right).
	\end{aligned}
	\]
	On the other hand, for arbitrary $0<\delta<1$,
	\[
	\begin{aligned}
	&(s-\tau)^{\alpha-1}\| T((t-\tau)^{\alpha}\theta)-T((s-\tau)^{\alpha}\theta)\|\\ &=(s-\tau)^{\alpha-1}\|\left[ T((t-\tau)^{\alpha}\theta-(s-\tau)^{\alpha}\theta)-I\right]A^{-\delta}A^{\delta}T((s-\tau)^{\alpha}\theta) \|\\
	&\le C_{\delta} (s-\tau)^{\alpha-1} \left[ (t-\tau)^{\alpha}\theta-(s-\tau)^{\alpha}\theta \right]^{\delta} (s-\tau)^{-\delta\alpha}\theta^{-\delta}\\
	&\le C_{\delta,\alpha} (s-\tau)^{\alpha(1-\delta)-1}(t-s)^{\alpha\delta}
	\end{aligned}
	\]
	where $C_{\delta,\alpha}$ depends on $\delta, \alpha$. Therefore,
	\[
	\begin{aligned}
	J_3&\le \frac{C_{\delta, \alpha}}{\Gamma(1+\alpha)} \int_{0}^s (s-\tau)^{\alpha(1-\delta)-1}d\tau(t-s)^{\alpha\delta}\\
	&\le \frac{C_{\delta, \alpha} T^{\alpha(1-\delta)}}{\alpha(1-\delta)\Gamma(1+\alpha)} (t-s)^{\alpha\delta}.
	\end{aligned}
	\]
\end{proof}

\begin{proof}[Proof of Theorem \ref{theo. main theorem1}]
	Let $u$ denotes the fixed-point of $S$. For any $f\in {\cal F}^{r,\beta}$, the existence and uniquness for the equation with external force term $f- B(u)\in {\cal F}^{r,\beta}$;
	\[
	D_t^{\alpha} (w-w(0)) +Aw =f- B(u), \ w(0)=u_0
	\]
	is assured by Proposition\ \ref{existence}. This solution $w$ coincides with $u$ because the representation formula of (\ref{mild sol}) necessarily be satisfied. Proposition\ \ref{prop. time regularity of S_0},\ \ref{prop. time regularity of S_1} provides the improvement of time regularity, $u\in C^{\beta}([0,T];X)$ when $u\in D(A^{q})$ and $f\in C([0,T];X)$. Suppose that $u\in D(A)$ and $f\in {\cal F}^{1,\beta}((0,T];X)$. It is sufficient to prove $Au\in C([0,T];X)$ because of (\ref{eq. inclusion of D(Balpha)}). The continuity of 
	\[
	AS_0(t)=\int_{0}^{\infty} h_{\alpha}(\theta) T(t^{\alpha}\theta)Au_0 d\theta
	\]
	at $t=0$ is obvious. It is sufficient to prove that for any $g\in {\cal F}^{1,\beta}((0,T];X)$,
	\[
	AS_1(g)(t)=\int_{0}^{t} \int_{0}^{\infty} (t-\tau)^{\alpha-1}\theta h_{\alpha}(\theta) AT((t-\tau)^{\alpha}\theta) g(\tau)d\tau d\theta\rightarrow 0
	\]
	as $t\rightarrow 0$. Indeed,
	\[
	\begin{aligned}
	AS_1(g)(t)&= \int_{0}^{t}\int_{0}^{\infty}(t-\tau)^{\alpha-1}\theta h_{\alpha}(\theta) AT((t-\tau)^{\alpha}\theta)\left[ g(\tau)-g(t)\right]d\tau d\theta\\
	&\hspace{1cm} +  \int_{0}^{t}\int_{0}^{\infty}(t-\tau)^{\alpha-1}\theta h_{\alpha}(\theta) AT((t-\tau)^{\alpha}\theta) g(t)d\tau d\theta\\
	&= \int_{0}^{t}\int_{0}^{\infty}(t-\tau)^{\alpha-1}\theta h_{\alpha}(\theta) AT((t-\tau)^{\alpha}\theta)\left[ g(\tau)-g(t)\right]d\tau d\theta+ \left[I-\Phi(t)\right]g(t)
	\end{aligned}
	\]
	See (\ref{eq. def of Phi}) for the definition of $\Phi(t)$. On the other hand, we estimate
	\[
	\begin{aligned}
	&\int_{0}^{t}\int_{0}^{\infty}(t-\tau)^{\alpha-1}\theta h_{\alpha}(\theta) \| AT((t-\tau)^{\alpha}\theta)\| \| g(\tau)-g(t)\|d\tau d\theta\\
	&\le C \int_{0}^{t} (t-\tau)^{\beta-1} \tau^{-\beta} d\tau\ w_g(t)\rightarrow 0\ \mbox{as}\ t\rightarrow 0
	\end{aligned}
	\]
	and
	\[
	\begin{aligned}
	\|g(t)-\Phi(t)g(t)\|&\le \|I-\Phi(t)\| \cdot\|g(t)-g(0)\| + \| g(0)- \Phi(t)g(0)\|\\
	&\le C\|g(t)-g(0) + \|g(0)-\Phi(t)g(0)\|\\
	&\rightarrow0\ \mbox{as}\ t\rightarrow 0
	\end{aligned}
	\]
	because of the strong continuity of $\Phi(t)$ at $t=0$.
\end{proof}

\begin{rem}\label{rem for non local perturb}
	Our results shall be extended to the non-local perturbation
	\[
	B:C([0,T];X)\rightarrow  C([0,T];X)
	\]
	such that
	\[
	\| B(u)-B(v) \|_{C([0,T];X)}\le L \| u-v \|_{C([0,T];X)}
	\]
	for some $L>0$ and maps the (weighted) H\"{o}lder space to itself.
\end{rem}

\section{Proof of Theorem \ref{theo. main theorem2}}\label{section. main theorem2 time diff}

In this section, we will prove Theorem\ \ref{theo. main theorem2}. If we differentiate both sides of (\ref{eq. semilinear problem}) formally by $t$, we should be able to derive the equation which $u'$ should satisfy. For preparation, we at first investigate evolution equations with integral initial conditions.

\subsection{Integral initial condition equation}
We investigate the following equation
\begin{equation}\label{initial cond with integral}
\left\{
\begin{aligned}
D_t^{\alpha} u +A u&=f,\ &&\mbox{in}\ (0,T],\\
(I^{\alpha-1}u)(0)&=x_0
\end{aligned}
\right.
\end{equation}
where
\begin{equation}\label{RL integral}
I^{\alpha-1}u(t)=\frac{1}{\Gamma(1-\alpha)}\int_{0}^{t}(t-\tau)^{-\alpha}u(\tau)d\tau.
\end{equation}
We assume that $f\in {\cal F}^{r,\beta}((0,T];X)$ for some $0<\beta<r\le 1$. The representation formula for mild solution shall be changed from (\ref{mild sol}).

\begin{prop}
	The solution of (\ref{initial cond with integral}) is given by the following;
	\begin{equation}\label{mild sol of integral problem}
	u(t)=\alpha\int_0^{\infty} h_{\alpha}(\theta)T(t^{\alpha}\theta)x_0 d\theta t^{\alpha-1} + \alpha \int_{0}^{t} \int_{0}^{\infty}  \theta h_{\alpha}(\theta)(t-\tau)^{\alpha-1}T((t-\tau)^{\alpha}\theta)f(\tau)d\tau d\theta
	\end{equation}
\end{prop}
\begin{proof}
	The formula (\ref{mild sol of integral problem}) can be deduced from (\ref{RL integral}) by the same way as in \cite{perturbation}. In fact, we apply the Laplace transform to (\ref{RL integral}) and use the fact (see, for instance \cite{Podlubny}),
	\[
	{\cal L}\left[ D_t^{\alpha} g \right](s)=s^{\alpha}{\cal L}[g](s) -I^{\alpha-1}g(0)
	\]
	where ${\cal L}$ is Laplace transform
	\[
	{\cal L}[g](s):=\int_{0}^{\infty} e^{-st}g(t)dt.
	\]
	Therefore, we denote $U={\cal L}[u], F={\cal L}[f]$, then,
	\[
	s^{\alpha}U(s)-x_0 +AU(s)=F(s)\ \Rightarrow\ U(s)= -(-s^{\alpha}-A)^{-1}(x_0+F(s)).
	\]
	We use the fact that for each $\lambda\le 0$,
	\[
	(\lambda-A)^{-1}=-\int_{0}^{\infty}e^{\lambda t} T(t)dt,
	\]
	Then,
	\[
	U(s)=\int_{0}^{\infty}e^{s^{\alpha}t}(x_0+F(s))dt.
	\]
	The rest of the calculations are the same as in \cite{perturbation}. The second term of (\ref{mild sol of integral problem}) is obviously the solution for 
	\[
	D_t^{\alpha}w +A w =f,\ I^{\alpha-1}(w)(0)=0.
	\]
	We denote the first term of (\ref{mild sol of integral problem}) by $w_0(t)$ and let
	\begin{equation}\label{eq. def of Phi}
		\Phi(t)= \int_{0}^{\infty} h_{\alpha}(\theta)T(t^{\alpha}\theta) d\theta.
	\end{equation}
	We observe 
	\[
	\begin{aligned}
	w_0(t) &= \int_{0}^{\infty} h(\theta)A^{-1}(\alpha t^{\alpha-1}\theta)A T(t^{\alpha}\theta)x_0d\theta\\
	&= -\int_{0}^{\infty} h(\theta) A^{-1} \frac{d}{dt}\left[ T(t^{\alpha}\theta)x_0 \right]d\theta\\
	&=-A^{-1}\frac{d}{dt}\Phi(t)x_0,
	\end{aligned}
	\]
	that is,
	\[
	Aw_0(t)=-\Phi'(t)x_0.
	\]
	On the other hands, (\ref{mild sol}) means that $\Phi(t)x_0$ is the solution for
	\[
	\frac{1}{\Gamma(1-\alpha)}\int_{0}^{t}(t-\tau)^{-\alpha}\Phi'(t)x_0d\tau= -A\Phi(t)x_0.
	\]
	Therefore,
	\[
	\begin{aligned}
	\frac{1}{\Gamma(1-\alpha)}\int_0^{t}(t-\tau)^{-\alpha}w_0(\tau)d\tau &= -A^{-1}\frac{1}{\Gamma(1-\alpha)}\int_{0}^{t} (t-\tau)^{-\alpha}\Phi'(t)x_0d\tau\\
	&= \Phi(t)x_0
	\end{aligned}
	\]
	which means,
	\[
	\frac{1}{\Gamma(1-\alpha)}\frac{d}{dt}\int_{0}^{t}(t-\tau)^{-\alpha}w_0(\tau)=\Phi'(t)x_0=-Aw_0(t).
	\]
	If $x_0\in \overline{D(A)}$,  the Lebesgue convergence theorem yields when $t\rightarrow 0$,
	\[
	\begin{aligned}
	\frac{\alpha}{\Gamma(1-\alpha)}\int_{0}^{t}(t-\tau)^{-\alpha}w_0(\tau)d\tau&=\frac{\alpha}{\Gamma(1-\alpha)}\int_{0}^{t}(t-\tau)^{-\alpha}\tau^{\alpha-1}\int_{0}^{\infty} \theta h(\theta)T(\tau^{\alpha}\theta)x_0d\tau d\theta\\
	&= \frac{\alpha}{\Gamma(1-\alpha)}\int_{0}^{1}(1-\tau)^{-\alpha}\tau^{\alpha-1}\int_{0}^{\infty} \theta h(\theta) T(\tau^{\alpha}\theta t^{\alpha})x_0 d\tau d\theta\\
	&\rightarrow \frac{\alpha}{\Gamma(1-\alpha)}\int_{0}^{1}(1-\tau)^{-\alpha}\tau^{\alpha-1} \int_{0}^{\infty}  \theta h(\theta ) x_0 d\tau d\theta\\
	&= \frac{\alpha}{\Gamma(1-\alpha)}B(1-\alpha,\alpha) \frac{\Gamma(2)}{\Gamma(1+\alpha)}x_0\\
	&=x_0
	\end{aligned}
	\]
\end{proof}
Next we investigate the following Lipschitz perturbation problem
\begin{equation}\label{perturbation with integral condition}
\left\{
\begin{aligned}
D_t^{\alpha} u +A u+B(u)&=f,\ &&\mbox{in}\ (0,T],\\
(I^{\alpha}u)(0)&=x_0.
\end{aligned}
\right.
\end{equation}

Proposition\ \ref{prop. time regularity of S_0} and the fact that $t\mapsto t^{\alpha-1}$ is in ${\cal F}^{\alpha-\gamma,\gamma}$ deduce the next lemma.

\begin{lem}\label{lem. continuity of w0}
	Suppose that $x_0\in \overline{D(A)}$. Then, $w_0\in {\cal F}^{\alpha-\gamma,\gamma}$ for any $0<\gamma <\alpha/2$.
\end{lem}

We define the map on ${\cal F}^{\alpha-\gamma,\gamma}$
\begin{equation}\label{eq. contraction map for integral initial cond}
	\begin{aligned}
	\tilde{S}:{\cal F}^{\alpha-\gamma, \gamma}\ni v\mapsto \tilde{S}(v)(t)&=\alpha\int_0^{\infty}\theta h_{\alpha}(\theta)T(t^{\alpha}\theta)x_0 d\theta t^{\alpha-1} \\
	&+\alpha \int_{0}^{t} \int_{0}^{\infty}  \theta h_{\alpha}(\theta)(t-\tau)^{\alpha-1}T((t-\tau)^{\alpha}\theta)(f(\tau)-B(v(\tau)))d\tau d\theta.
	\end{aligned}
\end{equation}

\begin{theo}\label{theo. perturbation with integral condition theorem}
	Let $f\in {\cal F}^{r,\beta}((0,T];X)$ such that $0<\beta<r\le 1$, $\beta<\alpha$, $0<\max(\beta,1-\alpha)<r$ and $\alpha/2\le r$. Then, (\ref{perturbation with integral condition}) has the unique solution
	\[
	u\in {\cal F}^{\alpha-\gamma, \gamma}((0,T];X)\cap C((0,T];D(A)).
	\]
	where $0<\gamma<\alpha/2$, $\alpha-\gamma \le r$ and $\gamma\le \beta$.
\end{theo}

\begin{proof}
We just need to show that (\ref{eq. contraction map for integral initial cond}) is contraction mapping on ${\cal F}^{\alpha-\gamma,\gamma}$. The condition for $\gamma$ and Proposition\ \ref{prop. time regularity of S_1} guarantee that $\tilde{S}$ maps ${\cal F}^{\alpha-\gamma,\gamma}$ to itself. Now, for each $v_1,v_2\in {\cal F}^{\alpha-\gamma, \gamma}$, we denote $\tilde{S}(v_1)=u_1, \tilde{S}(v_2)=u_2$. The difficulty of the proof is that we have to evaluate the seminorm part of $\|\cdot\|_{{\cal F}^{r,\beta}}$, not just uniform norm. For simplicity, we denote $\alpha-\gamma=\kappa$, also, we use $\lesssim$ as a evaluation by appropriate constant multiplication to avoid a complicated notation. We equip the norm
\[
|\|f\||_{\mu}=\sup_{0\le t\le T}e^{-\mu t}t^{1-\kappa}\| f(t)\| +\sup_{0\le s<t\le T} \frac{s^{1-\kappa+\gamma}e^{-\mu t}\| f(t)-f(s) \|}{(t-s)^{\gamma}}.
\]
for ${\cal F}^{\kappa,\gamma}((0,T];X)$ where $\mu>0$.

\vspace{2ex}\hspace{-3ex}{\large{\bf Estimate for uniform norm}}\\
	
We can easily see that
	\[
	\begin{aligned}
	\| u_1(t)-u_2(t)\|&\lesssim \int_{0}^{t}  e^{\mu t} e^{-\mu(t-\tau)}e^{-\mu \tau} (t-\tau)^{\alpha-1}\tau ^{\kappa-1}\tau^{1-\kappa}\| v_1(\tau)-v_2(\tau)\|d\tau\\
	e^{-\mu t}\| u_1(t)-u_2(t)\|&\lesssim \int_{0}^{t}  e^{-\mu(t-\tau)} (t-\tau)^{\alpha-1}\tau ^{\kappa-1}d\tau |\| v_1-v_2\||_{\mu}.
	\end{aligned}
	\]
	For sufficiently small $\epsilon>0$, Let $p=1+\epsilon$ and $q$ be the conjugate of $p$. Because of the H\"{o}lder's inequality
	\[
	\begin{aligned}
	e^{-\mu t}\| u_1(t)-u_2(t)\|&\lesssim \left( \int_{0}^{t}e^{-\mu q(t-\tau)} d\tau\right)^{1/q}\left( \int_{0}^{t} (t-\tau)^{p(\alpha-1)}\tau ^{p(\kappa-1)}d\tau \right)^{1/p}\cdot L|\| v_1-v_2\||_{\mu}\\
	&= B_{\mu} \underline{\left( \int_{0}^{t} (t-\tau)^{p(\alpha-1)}\tau ^{p(\kappa-1)}d\tau \right)^{1/p}}\cdot |\| v_1-v_2\||_{\mu}
	\end{aligned}
	\]
	where 
	\[
	\begin{aligned}
	\left(\int_{0}^{t}e^{-\mu q (t-\tau)}d\tau\right)^{1/q}&= \left(\int_{0}^{t}e^{-\mu q \tau}d\tau\right)^{1/q}\\
	&\le \left(\int_{0}^{T}e^{-\mu q \tau}d\tau\right)^{1/q}= B_{\mu}\rightarrow 0,\ \mu\rightarrow \infty
	\end{aligned}
	\]
	The underlined part shall be estimated as follows;
	\[
	\begin{aligned}
	\text{(underline part)}^p & = \int_{0}^t  t^{p(\alpha-1)} \left( 1-\frac{\tau}{t} \right)^{p(\alpha-1)} t^{p(\kappa-1)}\left( \frac{\tau}{t} \right)^{p(\kappa-1)}d\tau\\
	&= t^{p(\alpha+\kappa-2)} \int_{0}^{t}  \left( 1-\frac{\tau}{t} \right)^{p(\alpha-1)}\left( \frac{\tau}{t} \right)^{p(\kappa-1)}d\tau,\ \ (\tau\leftrightarrow\tau/t,\ d\tau\leftrightarrow td\tau)\\
	&=  t^{p(\alpha+\kappa-2)+1}\in
	t_{0}^1  (1-\tau)^{p(\alpha-1)} \tau ^{p(\kappa-1)}d\tau\\
	&= t^{p(\alpha+\kappa-2)+1} B(p(\alpha-1)+1, p(\kappa-1)+1)\\
	\text{(underline part)}& = B(p(\alpha-1)+1, p(\kappa-1)+1)^{1/p}\ t^{\alpha-1+1/p}\ t^{\kappa-1}.
	\end{aligned}
	\]
	That is,
	\[
	e^{-\mu t}t^{1-\kappa}\| u_1(t)-u_2(t)\|\lesssim B_{\mu}t^{\alpha-1+1/p} |\|v_1-v_2\||_{\mu}
	\]
	We can choose the $\epsilon>0$ such that the indices meet
	\[
	\left\{
	\begin{aligned}
	&\alpha-1+\frac{1}{p}=\alpha-\frac{\epsilon}{1+\epsilon}>0,\\
	&p(\alpha-1)+1 =(1+\epsilon)\alpha -\epsilon>0,\\
	&p(\kappa-1)+1=(1+\epsilon)\kappa-\epsilon>0.
	\end{aligned}
	\right.
	\]
	Hence, we conclude
	\[
	\sup_{0\le t\le T}e^{-\mu t}t^{1-\kappa}\| u_1(t)-u_2(t)\|\le B_{\mu}C(L,T) |\|v_1-v_2\||_{\mu}
	\]
	
\vspace{2ex}\hspace{-3ex}{\large{\bf Estimate for seminorm}}\\

	We set $V= B(v_1)-B(v_2)$ and $U=\tilde{S}v_1-\tilde{S}v_2$. We shall estimate $\| U(t)-U(s)\|$. The simple calculation yields
	\[
	\begin{aligned}
	&\| U(t)-U(s)\|\\
	&\le \left\| \int_{0}^{t}(t\text{'s formula}) \pm\int_{0}^{s}(t\text{'s formula}) -\int_{0}^{s}(s\text{'s formula})\right\|\\
	&\le \int_{s}^{t} \int_{0}^{\infty} \theta h_{\alpha}(\theta) (t-\tau)^{\alpha-1} \| T((t-\tau)^{\alpha}\theta)\|_{B(X)} \tau^{\kappa-1}\tau^{1-\kappa}e^{\mu t}e^{-\mu (t-\tau)}e^{-\mu \tau}\|V(\tau)\|d\tau d\theta\\
	&\hspace{1ex}+\int_{0}^{s}\int_{0}^{\infty}\theta h_{\alpha}(\theta)\| (s-\tau)^{\alpha-1} T((s-\tau)^{\alpha}\theta)- (t-\tau)^{\alpha-1}T((t-\tau)^{\alpha}\theta)\|\tau^{\kappa-1}\tau^{1-\kappa}e^{\mu t}e^{-\mu (t-\tau)}e^{-\mu \tau}\| V(\tau)\|d\tau d\theta
	\end{aligned}
	\]
	\[
	\begin{aligned}
	&e^{-\mu t}\|U(t)-U(s)\|\\
	&\lesssim \int_{s}^{t} (t-\tau)^{\alpha-1}\tau^{\kappa-1}e^{-\mu (t-\tau)}d\tau |\| v_1-v_2\||_{\mu}\\
	&\hspace{1.5ex}+ \int_{0}^s \int_{0}^{\infty} \theta h_{\alpha}(\theta) (s-\tau)^{\alpha-1}\tau^{\kappa-1}e^{-\mu (s-\tau)}\| T((t-\tau)^{\alpha}\theta)- T((s-\tau)^{\alpha}\theta) \|_{B(X)}d\tau d\theta  |\|v_1-v_2\||_{\mu}\\
	&\hspace{1.5ex}+ \int_{0}^{s}  \left| (s-\tau)^{\alpha-1}- (t-\tau)^{\alpha-1} \right| \tau^{\kappa-1} e^{-\mu (s-\tau)}d\tau d\theta  |\|v_1-v_2\||_{\mu}\\
	&=: (K_1+K_2+K_3)|\|v_1-v_2\||_{\mu}.
	\end{aligned}
	\]
	We shall estimate $s^{1-\kappa+\gamma}K_i$ from above by the product of the term which converges to zero and the power of $(t-s)$.
	
\vspace{2ex}\hspace{-3ex}{\bf Estimate for $K_1$}
	
	\[
	K_1=\int_{s}^{t} (t-\tau)^{\alpha-1}\tau^{\kappa-1}e^{-\mu (t-\tau)}d\tau.
	\]
	Let $p=1+\epsilon$, $q$ be its conjugate. Then, by the H\"{o}lder's inequality 
	\[
	\begin{aligned}
	K_1 &=\int_{s}^{t}(t-\tau)^{\alpha-1}\tau^{\gamma}\tau^{\kappa-\gamma-1}e^{-\mu (t-\tau)}d\tau\\
	&\le \int_s^t (t-\tau)^{\alpha-1}e^{-\mu(t-\tau)}d\tau\ t^{\gamma}s^{\kappa-\gamma-1}\\
	s^{1-\kappa+\gamma}K_1&\lesssim \int_s^t(t-\tau)^{\alpha-1}e^{-\mu(t-\tau)}d\tau\ t^{\gamma}\\
	&\lesssim \left( \int_s^t e^{-\mu(t-\tau)q}d\tau \right)^{1/q}\left( \int_s^t (t-\tau)^{(\alpha-1)p}d\tau \right)^{1/p}t^{\gamma}\\
	&\le C_{\mu} (t-s)^{\alpha-1+1/p}t^{\gamma}
	\end{aligned}
	\]
	where
	\[
	\left( \int_{0}^{t-s}  e^{-\mu \tau q} d\tau\right)^{1/q}\le \left( \int_{0}^{T}  e^{-\mu \tau q} d\tau\right)^{1/q}= C_{\mu}\rightarrow 0,\ \mu\rightarrow \infty.
	\]
	We choose $0<\epsilon<\alpha/(1-\alpha)$. Hence, 
	\[
	s^{1-\kappa+\gamma}K_1\lesssim C_{\mu} (t-s)^{\alpha-1+1/p}t^{\gamma}.
	\]
	
\vspace{2ex}\hspace{-3ex}{\bf Estimate for $K_2$}
	
	\[
	K_2 = \int_{0}^{s}  \int_{0}^{\infty}  \theta h_{\alpha}(\theta)\underline{(s-\tau)^{\alpha-1}\tau^{\kappa-1} e^{-\mu (s-\tau)}\| T((t-\tau)^{\alpha}\theta)-T((s-\tau)^{\alpha}\theta) \|}d\tau d\theta.
	\]
	We select arbitrary $0<\delta<1$. The underlined part shall be estimated;
	\[
	\begin{aligned}
	\text{(underline part)} &= (s-\tau)^{\alpha-1}\tau^{\kappa-1} e^{-\mu (s-\tau)}\|\left[ T((t-\tau)^{\alpha}\theta-(s-\tau)^{\alpha}\theta)-I\right]A^{-\delta}A^{\delta}T((s-\tau)^{\alpha}\theta) \|\\
	&\lesssim (s-\tau)^{\alpha-1}\tau^{\kappa-1} e^{-\mu (s-\tau)} \left[ (t-\tau)^{\alpha}\theta-(s-\tau)^{\alpha}\theta \right]^{\delta} (s-\tau)^{-\delta\alpha}\theta^{-\delta}\\
	&= e^{-\mu (s-\tau)}(s-\tau)^{\alpha(1-\delta)-1}\tau^{\kappa-1}(t-s)^{\alpha\delta}
	\end{aligned}
	\]
	We select $p=1+\epsilon$ for some $\epsilon>0$, and $q$ be the conjugate of $p$ as usual. Therefore
	\[
	\begin{aligned}
	K_2&\lesssim \int_{0}^s  e^{-\mu (s-\tau)}(s-\tau)^{\alpha(1-\delta)-1}\tau^{\kappa-1}(t-s)^{\alpha\delta}d\tau\\
	&\le \left( \int_{0}^s  e^{-\mu (s-\tau)q} d\tau\right)^{1/q}\left( \int_{0}^s  (s-\tau)^{\{\alpha(1-\delta)-1\}p} \tau^{(\kappa-1)p}d\tau \right)^{1/p}(t-s)^{\alpha\delta}\\
	&\le D_{\mu} \left( \underline{\int_{0}^s  (s-\tau)^{\{\alpha(1-\delta)-1\}p} \tau^{(\kappa-1)p} }d\tau\right)^{1/p}(t-s)^{\alpha\delta}
	\end{aligned}
	\]
	where
	\[
	\left( \int_{0}^s d e^{-\mu (s-\tau)q}\tau \right)^{1/q}\le \left( \int_{0}^T  e^{-\mu \tau q}d\tau \right)^{1/q}=D_{\mu}\rightarrow 0,\ \mu\rightarrow \infty.
	\]
	Moreover,
	\[
	\begin{aligned}
	\text{(underline part)}&=\int_{0}^{s}\ s^{\{\alpha(1-\delta)-1\}p} \left( 1-\frac{\tau}{s} \right)^{\{\alpha(1-\delta)-1\}p} s^{(\kappa-1)p}\left( \frac{\tau}{s} \right)^{(\kappa-1)p}d\tau\\
	&=s^{\star} \int_{0}^1  (1-\tau)^{\{\alpha(1-\delta)-1\}p}\tau^{(\kappa-1)p}d\tau\\
	&=s^{\star} B(\{\alpha(1-\delta)-1\}p+1,(\kappa-1)p+1)
	\end{aligned}
	\]
	where $\star= \{\alpha(1-\delta)-1\}p +(\kappa-1)p+1$. Thus,
	\[
	\begin{aligned}
	K_2&\lesssim D_{\mu}(t-s)^{\alpha\delta} B(\left\{ \alpha(1-\delta)-1 \right\} p+1, (r-1)p+1)^{1/p}\ s^{\alpha(1-\delta)-1+\gamma+1/p}\ s^{\kappa-1-\gamma}\\
	s^{1-\kappa+\gamma}K_2&\lesssim D_{\mu}(t-s)^{\alpha\delta}B(\left\{ \alpha(1-\delta)-1 \right\} p+1, (\kappa-1)p+1)^{1/p}\ s^{\alpha(1-\delta)-1+\gamma+1/p}.
	\end{aligned}
	\]
	We choose appropriate $\epsilon>0$ such that
	\[
	\left\{
	\begin{aligned}
	&\alpha(1-\delta)-1+\gamma+\frac{1}{p}=\alpha(1-\delta) +\gamma -\frac{\epsilon}{1+\epsilon}>0,\\
	&\{\alpha(1-\delta)-1\}p+1 = (1+\epsilon)\alpha(1-\delta)-\epsilon>0,\\
	&(\kappa-1)p+1 =\kappa(1+\epsilon)-\epsilon>0.
	\end{aligned}
	\right.
	\]
	Hence,
	\[
	s^{1-\kappa+\gamma}K_2\lesssim D_{\mu} (t-s)^{\alpha\delta}s^{\alpha(1-\delta)-1+\gamma+1/p}
	\]
	
\vspace{2ex}\hspace{-3ex}{\bf Estimate for $K_3$}
	
	\[
	K_3 = \int_{0}^{s}  \left| (s-\tau)^{\alpha-1}- (t-\tau)^{\alpha-1} \right| \tau^{\kappa-1} e^{-\mu (s-\tau)}d\tau d\theta
	\]
	We note that $\tau\mapsto \tau^{\alpha-1}$ is in ${\cal F}^{\alpha-\sigma, \sigma}$ for all $\sigma <\alpha/2$, then,
	\[
	K_3 \lesssim (t-s)^{\sigma}\underline{\int_{0}^{s} (s-\tau)^{\alpha-1-2\sigma} \tau^{\kappa-1}e^{-\mu (s-\tau)}d\tau}.
	\]
	As before, Let $p=1+\epsilon$, and $q$ be its conjugate,
	\[
	\begin{aligned}
	\text{(underline part)} &\le \left( \int_{0}^{s}e^{-\mu(s-\tau)q}d\tau \right)^{1/q}\left( \int_{0}^{s}(s-\tau)^{(\alpha-1-2\sigma)p}\tau^{(\kappa-1)p}d\tau \right)^{1/p}\\
	&\le E_{\mu} s^{\alpha-1-2\sigma+\gamma+1/p} s^{\kappa-1-\gamma} B((\alpha-1-2\sigma)p+1, (\kappa-1)p+1)^{1/p}
	\end{aligned}
	\]
	where
	\[
	E_{\mu}=\left( \int_{0}^T e^{-\mu \tau p}d\tau \right)^{1/p}.
	\]
	That is,
	\[
	s^{1-\kappa+\gamma} K_3\lesssim E_{\mu} (t-s)^{\sigma}  s^{\alpha-1-2\sigma+\gamma+1/p}.
	\]
	We select $\sigma <\alpha/2$ and $\epsilon>0$ such that
	\[
	\left\{
	\begin{aligned}
	& \alpha-1-2\sigma +\gamma +\frac{1}{p}=\alpha+\gamma-2\sigma -\frac{\epsilon}{1+\epsilon}>0,\\
	&(\alpha-1-2\sigma)p+1>0,\\
	&(\kappa-1)p+1>0.
	\end{aligned}
	\right.
	\]
	Finally, we select again $\epsilon$, $\delta$ and $\sigma$ such that
	\[
	\frac{\alpha}{2}\le \min\left(\alpha-\frac{\epsilon}{1+\epsilon},\ \alpha\delta \right),\ \ \gamma\le \sigma<\frac{\alpha}{2}
	\]
	
\end{proof}

\begin{proof}[Proof of Theorem \ref{theo. main theorem2}]
	
	Let $r, \beta$ and $r',\beta'$ satisfies the conditions as in Theorem\ \ref{theo. main theorem1}, Theorem\ \ref{theo. perturbation with integral condition theorem} respectively.  For given $u\in C([0,T];X)$, let us consider the solution for 
	\begin{equation}
	\left\{
	\begin{aligned}
	D_t^{\alpha} w +A w+B'(u)w&=f',\ &&\mbox{in}\ (0,T],\\
	(I^{\alpha}w)(0)&=x_0,
	\end{aligned}
	\right.
	\end{equation}
	also, let
	\[
	U(t)=\int_{0}^{t}w(\tau)d\tau+u_0.
	\]
	Theorem\ \ref{theo. perturbation with integral condition theorem} guarantees the well-definedness of $u\mapsto U$. We shall use the Banach fixed point theorem for the map $C([0,T];X)\ni u\mapsto U\in C([0,T];X)$. We equip the norm
	\[
	|\|u\||:=\sup_{0\le t\le T} e^{-\mu t}\|u(t)\|
	\]
	to $C([0,T];X)$. The representation (\ref{mild sol of integral problem}) and the assumption $f'\in {\cal F}^{r,\beta}$ gives the estimate
	\[
	\begin{aligned}
	\|w(t)\| &\le C\left( t^{\alpha-1} + t^{\alpha+r-1} + \int_{0}^{t} (t-\tau)^{\alpha-1}\|w(\tau)\|d\tau\right)\\
	&\le C\left( t^{\alpha-1} +\int_{0}^{t} (t-\tau)^{\alpha-1}\|w(\tau)\|d\tau\right)
	\end{aligned}
	\]
	for some constant $C$ depends on the semigroup, existence time $T>0$, $f'$ and $L$. Therefore, Corollary\ \ref{cor volterra} yields
	\[
	\|w(t)\| \lesssim  t^{\alpha-1} (1+b't)^{2-\alpha} e^{b't+1}\lesssim t^{\alpha-1}.
	\]
	where $b'=(C\Gamma(\alpha))^{1/\alpha}$. For given $u_1,u_2$,
	\[
	\begin{aligned}
	w_1(t)-w_2(t) &= \int_{0}^{t}\int_{0}^{\infty} \theta h_{\alpha}(\theta) (t-\tau)^{\alpha-1}T((t-\tau)^{\alpha}\theta)\left( B'(u_1(\tau))w_1(\tau)-B'(u_2(\tau))w_2(\tau) \right) d\tau d\theta\\
	\|w_1(t)-w_2(t)\| &\lesssim \int_{0}^{t} (t-\tau)^{\alpha-1} \| B'(u_1(\tau))-B'(u_2(\tau))\| \|w_2(\tau)\| d\tau\\
	&\hspace{5cm}+\int_{0}^{t} (t-\tau)^{\alpha-1} \|B'(u_1(\tau))\| \|w_1(\tau)-w_2(\tau)\|d\tau\\
	&\lesssim \int_{0}^{t} (t-\tau)^{\alpha-1} \tau^{\alpha-1} e^{\mu t}e^{-\mu(t-\tau)}e^{-\mu \tau}\|u_1(\tau)-u_2(\tau)\|d\tau\\
	&\hspace{5cm}+\int_{0}^{t} (t-\tau)^{\alpha-1}e^{\mu t}e^{-\mu (t-\tau)}e^{-\mu \tau} \|w_1(\tau)-w_2(\tau)\|d \tau\\
	e^{-\mu t}\|w_1(t)-w_2(t)\|&\lesssim t^{2\alpha-1} |\|u_1-u_2\|| +\int_{0}^{t} (t-\tau)^{\alpha-1} e^{-\mu \tau}\|w_1(\tau)-w_2(\tau)\|d\tau 
	\end{aligned}
	\]
	Therefore, we use again Corollary\ \ref{cor volterra} for $e^{-\mu t}\|w_1(t)-w_2(t)\|$ to obtain
	\[
	\begin{aligned}
	\|w_1(t)-w_2(t)\|\lesssim e^{\mu t} t^{2\alpha-1} (1+bt)^{2-2\alpha} |\|u_1-u_2\||
	\end{aligned}
	\]
	for appropriate $b>0$. Hence,
	\[
	\begin{aligned}
	\| U_1(t)-U_2(t)\| &\lesssim\int_{0}^{t} e^{\mu \tau} \tau^{2\alpha-1}\left( 1+b\tau \right)^{2-2\alpha} d\tau\ |\|u_1-u_2\||\\
	e^{-\mu t}\|U_1(t)-U_2(t)\|&\lesssim \int_{0}^{t}e^{-\mu (t-\tau)} \tau^{2\alpha-1}\left( 1+b\tau \right)^{2-2\alpha} d\tau\ |\|u_1-u_2\||
	\end{aligned}
	\]
	For small $\epsilon>0$, we set $p=1+\epsilon$ and its conjugate $q$. Since $(1+bt)^{2-2\alpha}\le (1+bT)^{2-2\alpha}$,
	\[
	\begin{aligned}
	\\
	e^{-\mu t}\|U_1(t)-U_2(t)\| &\lesssim\int_{0}^{t}e^{-\mu (t-\tau)} \tau^{2\alpha-1}d\tau |\|u_1-u_2\||\\
	&\lesssim \left( \int_{0}^{t}e^{-\mu(t-\tau)q}d\tau \right)^{1/q}\left( \int_{0}^{t}\tau^{(2\alpha-1)p}d\tau \right)^{1/p}|\|u_1-u_2\||.\\
	\end{aligned}
	\]
	We choose $\epsilon>0$ such that
	\[
	(2\alpha-1)p+1=(2\alpha-1)(1+\epsilon)+1=2\alpha+(2\alpha-1)\epsilon>0.
	\]
	So we can take sufficiently large $\mu $ so that $u\mapsto U$ becomes contraction mapping. Then, the fixed-point $u$ meets the following properties;
	\[
	\left\{
	\begin{aligned}
	\frac{1}{\Gamma(1-\alpha)} \frac{d}{dt}\int_{0}^{t}(t-\tau)^{-\alpha} u'(\tau)d\tau +Au'(t) +B'(u(t))u'(t)=f'(t),\\
	\frac{1}{\Gamma(1-\alpha)}\int_{0}^{t}(t-\tau)^{-\alpha}u'(\tau)d\tau \rightarrow x_0.
	\end{aligned}
	\right.
	\]
	Integrating both sides, we deduce
	\[
	\frac{1}{\Gamma(1-\alpha)}\int_{0}^{t}(t-\tau)^{-\alpha} u'(\tau)d\tau-x_0 +A(u(t)-u_0) +B(u(t))-B(u_0)=f(t)-f(0).
	\]
	If we choose
	\[
	x_0=-Au_0 -B(u_0)+f(0),
	\]
	the fixed-point $u$ coincides with the unique solution of 
	\[
	\left\{
	\begin{aligned}
	D_t^{\alpha} (u-u_0) +Au +B(u)&=f,\\
	u(0)&=u_0.
	\end{aligned}
	\right.
	\]
	This completes the proof.
\end{proof}

\section{Proof of Theorem \ref{theo. main theorem3}}\label{section. main theorem3 locally lipschitz perturbation}

In this section, we investigate the locally Lipschitz perturbation problem. For the discussion on the extension of the solution, we refer to \cite{fujita eq}, \cite{GL equation}.

\begin{proof}[Proof of Theorem \ref{theo. main theorem3}]
	We begin with the construction of a local solution. Let
	\[
	\tilde{B}(t,u)=\left\{
	\begin{aligned}
	&B(t,u)\ &&\mbox{if}\ \|u\|\le M,\\
	&B\left(t, \frac{M}{\|u\|}u \right)\ &&\mbox{if}\ \|u\|>M.
	\end{aligned}
	\right.
	\]
	We can easily see that $\tilde{B}$ is globally Lipschitz continuous. Therefore, there exists a unique solution $u$ for
	\[
	D_t^{\alpha}(u-u_0)+Au+\tilde{B}(u)=f,\ \ u(0)=u_0.
	\]
	We take sufficiently large $M$ such that $\|u_0\|<M/2$, then, the continuity of $u$ guarantees the existence of $T>0$ such that 
	\[
	\|u(t)\|<M\ \ \forall t\in [0,T].
	\]
	It means that $u$ satisfies (\ref{eq. semilinear problem}) on $[0,T]$. We assume $T^*<\infty$, and there exists $C>0$ such that for any $\epsilon>0$, we can take $t_{\epsilon}\in (T^*-\epsilon, T^*)$ such that $\|u(t_{\epsilon})\|\le C$. We set 
	\[
	\begin{aligned}
	E_{h, \delta}&:=\{ v\in C^{\beta}([t_{\epsilon}, t_{\epsilon}+h];X); \| v(t_{\epsilon})\|=\|u(t_{\epsilon})\|,\ \|v-u(t_{\epsilon})\|_{C([t_{\epsilon}t_{\epsilon}+h];X)}\le \delta \},\\
	\|\cdot\|_{E_{h,\delta}}&=\|\cdot\|_{C^{\beta}([t_{\epsilon},t_{\epsilon}+h];X)}
	\end{aligned}
	\]
	and for each $v\in E_{h, \delta}$, $t\in [t_{\epsilon}, t_{\epsilon}+h]$,
	\[
	\begin{aligned}
	{\cal T}(v)(t):= \int_{0}^{\infty}h_{\alpha}(\theta)T(t^{\alpha}\theta) u_0 d\theta &+\alpha\int_{0}^{t_{\epsilon}}\int_{0}^{\infty} (t-\tau)^{\alpha-1}\theta h_{\alpha}(\theta)T((t-\tau)^{\alpha}\theta)\left[ f(\tau)-B(u(\tau))\right]d\tau d\theta\\
	&+\alpha\int_{t_{\epsilon}}^{t}\int_{0}^{\infty}(t-\tau)^{\alpha-1}\theta h_{\alpha}(\theta) T((t-\tau)^{\alpha}\theta)\left[ f(\tau)-B(v(\tau)) \right] d\tau d\theta.
	\end{aligned}
	\]
	For any $v\in E_{h,\delta}$, the representation of mild solution povides ${\cal T}(v)(t_{\epsilon})=u(t_{\epsilon})$ and 
	\[
	\begin{aligned}
	&\|{\cal T}(v)(t)-u(t_{\epsilon})\|\\
	&=\|{\cal T}(v)(t)-{\cal T}(v)(t_{\epsilon})\|\\
	&=\int_{0}^{\infty} h_{\alpha}(\theta) \left[ T(t^{\alpha}\theta)-T(t_{\epsilon}^{\alpha}\theta) \right]u_0d\theta\\
	&\hspace{2ex} + \alpha\int_{0}^{t_{\epsilon}} \int_{0}^{\infty} \theta h_{\alpha}(\theta)\left[ (t-\tau)^{\alpha-1}T((t-\tau)^{\alpha}\theta)- (t_{\epsilon}-\tau)^{\alpha-1}T((t_{\epsilon}-\tau)^{\alpha}\theta)\right] \left[ f(\tau)-B(u(\tau))\right]d\tau d\theta\\
	&\hspace{2ex}+\alpha\int_{t_{\epsilon}}^{t}\int_{0}^{\infty}(t-\tau)^{\alpha-1}\theta h_{\alpha}(\theta) T((t-\tau)^{\alpha}\theta)\left[ f(\tau)-B(v(\tau)) \right] d\tau d\theta.\\
	\end{aligned}
	\]
	As in the proof of Theorem\ \ref{theo. main theorem1}, we estimate $\|{\cal T}(v)(t)-u(t_{\epsilon})\|\lesssim (t-t_{\epsilon})^{\beta}$. 
	This inequality yields the existence of $\tilde{h}>0$ such that if $h<\tilde{h}$, then $\|{\cal T}(v)(t)-u(t_{\epsilon})\|\le \delta$, i.e., ${\cal T}$ maps $E_{h,\delta}$ to itself. We note that we can choose $\tilde{h}$ independently of $\epsilon$. For arbitrary $v_1,v_2\in E_{h, \delta}$, $\|B(v_1(t))-B(v_2(t))\|\le L (u(t_{\epsilon})+\delta)\|v_1(t)-v_2(t)\|$. We write in abbreviated form of $L(u(t_{\epsilon})+\delta)\le L(C+\delta)=:L$.
	\[
	\begin{aligned}
	\|{\cal T}(v_1)(t)-{\cal T}(v_2)(t)\| &\le \frac{LC_1}{\Gamma(\alpha)} \int_{t_{\epsilon}}^{t} (t-\tau)^{\alpha-1} \| v_1(\tau)-v_2(\tau) \|d\tau\\
	&\le \frac{LC_1h^{\alpha}}{\alpha\Gamma(\alpha)} \|v_1-v_2\|_{E_{h,\delta}}.
	\end{aligned}
	\]
	Also, let $U={\cal T}(v_1)-{\cal T}(v_2)$, then, for each $t_{\epsilon}\le s<t\le t_{\epsilon}+h$, 
	\[
	\begin{aligned}
	\alpha^{-1}(U(t)-U(s))&= \int_{t_{\epsilon}}^{t}(t\text{'s formula}) \pm \int_{t_{\epsilon}}^{s}(t\text{'s formula})-\int_{t_{\epsilon}}^{s}(s\text{'s formula})\\
	&=\int_{s}^{t}(t\text{'s formula}) + \int_{t_{\epsilon}}^{s} \left[ (t\text{'s formula})-(s\text{'s formula}) \right]
	\end{aligned}
	\]
	\[
	\begin{aligned}
	\|U(t)-U(s)\| &\le \frac{C_1 L}{\Gamma(\alpha)}\int_{s}^{t} (t-\tau)^{\alpha-1} \|v_1(\tau)-v_2(\tau)\|d\tau\\
	&\hspace{2ex}+ \frac{C_1 L}{\Gamma(\alpha)}\int_{t_{\epsilon}}^{s} \left| (s-\tau)^{\alpha-1}-(t-\tau)^{\alpha-1} \right| \|v_1(\tau)-v_2(\tau)\|d\tau\\
	&\hspace{2ex} + \alpha L\int_{t_{\epsilon}}^{s}\int_{0}^{\infty} \theta h_{\alpha}(\theta) (s-\tau)^{\alpha-1}\left\| T((t-\tau)^{\alpha}\theta)-T((s-\tau)^{\alpha}\theta) \right\| \|v_1(\tau)-v_2(\tau)\|d\theta\\
	&\le (J_1+J_2+J_3)\|v_1-v_2\|_{E_{h,\delta}}
	\end{aligned}
	\]
	where
	\[
	\begin{aligned}
	J_1&= \frac{C_1 L}{\Gamma(\alpha)}\int_{s}^{t} (t-\tau)^{\alpha-1}d\tau,\\
	J_2&= \frac{C_1 L}{\Gamma(\alpha)}\int_{t_{\epsilon}}^{s} \left| (s-\tau)^{\alpha-1}-(t-\tau)^{\alpha-1} \right|d\tau,\\
	J_3&= \alpha L\int_{t_{\epsilon}}^{s}\int_{0}^{\infty} \theta h_{\alpha}(\theta) (s-\tau)^{\alpha-1}\left\| T((t-\tau)^{\alpha}\theta)-T((s-\tau)^{\alpha}\theta) \right\|d\tau.
	\end{aligned}
	\]
	We estimate
	\[
	\|J_1\| \le \frac{C_1L}{\alpha\Gamma(\alpha)} (t^{\alpha}-s^{\alpha})\le \frac{C_{\alpha}C_1L}{\alpha\Gamma(\alpha)}h^{\alpha-\beta}(t-s)^{\beta}.
	\]
	where $C_{\alpha}$ is the H\"{o}lder norm of $\tau \mapsto \tau^{\alpha}$. Furthermore,
	\[
	\begin{aligned}
	\|J_2\|&=\frac{C_1L}{\alpha\Gamma(\alpha)} \left[ (s-t_{\epsilon})^{\alpha}-(t-t_{\epsilon})^{\alpha}+(t-s)^{\alpha} \right]\\
	&\le \frac{2C_{\alpha}C_1L}{\alpha\Gamma(\alpha)} h^{\alpha-\beta}(t-s)^{\beta}
	\end{aligned}
	\]
	Also, we take auxiliary number $0<\eta<1$ to obtain
	\[
	\begin{aligned}
	\|J_3\| &\le \frac{\alpha C_{\alpha,\eta} L\Gamma(2-\eta)}{\Gamma(1+\alpha(1-\eta))}\int_{t_{\epsilon}}^{s}(s-\tau)^{\alpha(1-\eta)-1}(t-s)^{\alpha\eta}\\
	&\le \frac{C_{\alpha,\eta} L\Gamma(2-\eta)}{\Gamma(1+\alpha(1-\eta))} h^{\alpha(1-\eta)}(t-s)^{\alpha\eta}
	\end{aligned}
	\]
	where $C_{\alpha,\eta}$ is appropriate constant depends only on $\alpha, \eta$. We take $\eta$ such that $\beta= \alpha\eta$. Therefore, we choose sufficiently small $h$ such that
	\[
	h^{\alpha-\beta}< \frac{1}{2}\min\left( \frac{\alpha\Gamma(\alpha)}{C_1L},\ \frac{\alpha \Gamma(\alpha)}{C_{\alpha}C_1L},\ \frac{\alpha\Gamma(\alpha)}{2C_{\alpha}C_1L},\ \frac{\Gamma(1+\alpha(1-\eta))}{C_{\alpha,\eta}L\Gamma(2-\eta)} \right)\ \mbox{and}\ h<\tilde{h}
	\]
	in order to make ${\cal T}$ be a contraction. Then, its fixed point $v={\cal T}v$ satisfies
	\[
	\begin{aligned}
	v(t)= \int_{0}^{\infty}h_{\alpha}(\theta)T(t^{\alpha}\theta) u_0 d\theta &+\alpha\int_{0}^{t_{\epsilon}}\int_{0}^{\infty} (t-\tau)^{\alpha-1}\theta h_{\alpha}(\theta)T((t-\tau)^{\alpha}\theta)\left[ f(\tau)-B(u(\tau))\right]d\tau d\theta\\
	&+\alpha\int_{t_{\epsilon}}^{t}\int_{0}^{\infty}(t-\tau)^{\alpha-1}\theta h_{\alpha}(\theta) T((t-\tau)^{\alpha}\theta)\left[ f(\tau)-B(v(\tau)) \right] d\tau d\theta.
	\end{aligned}
	\]
	Let
	\[
	w(t)=\left\{
	\begin{aligned}
	&u(t)\ &&\mbox{for}\ t\in [0,t_{\epsilon}],\\
	&v(t)\ &&\mbox{for}\ t\in [t_{\epsilon},t_{\epsilon}+h],
	\end{aligned}
	\right.
	\]
	then $w$ satisfies for all $t\in [0, t_{\epsilon}+h]$
	\[
	w(t)=\int_{0}^{\infty}h_{\alpha}(\theta)T(t^{\alpha}\theta) u_0 d\theta + +\alpha\int_{0}^{t}\int_{0}^{\infty} (t-\tau)^{\alpha-1}\theta h_{\alpha}(\theta)T((t-\tau)^{\alpha}\theta)\left[ f(\tau)-B(w(\tau))\right]d\tau d\theta,
	\]
	this means that $w$ satisfies 
	\[
	D_{t}^{\alpha}(w-u_0)+Aw +B(w)=f\ \mbox{in}\ (0,t_{\epsilon}+h],\ w(0)=u_0
	\]
	because of the H\"{o}lder continuity of $\tau \mapsto B(w(\tau))$. Unrelation between $\epsilon>0$ and $h>0$ enables us to take $\epsilon<h$ and deduce contradiction.
\end{proof}

\section{Application}\label{section. application}

Let us consider the following combustion equation
\begin{equation}\label{combustion}
D_t^{\alpha}(u-u_0)-\Delta u - e^{-1/u}\chi_{\{u>0\}}=f(t).
\end{equation}
This model describes some reaction-diffusion processes with the generation of heats, based on the Arrhenius law which states the reaction rates is proportional to $e^{-C/u}$ for some $C>0$, where $u$ represents the absolute temperature \cite{combustion}. The perturbation part
\[
r\mapsto \left\{
\begin{aligned}
& e^{-1/r}\ &&\mbox{if}\ r>0,\\
& 0\ &&\mbox{if}\ r\le 0
\end{aligned}
\right.
\]
is obviously Lipschitz, so we apply Theorem\ \ref{theo. main theorem1} and Theorem\ \ref{theo. main theorem2} with the following framework;
\begin{equation}\label{continuous space}
\begin{aligned}
&X=u\in C(\overline{\Omega}),\\
&D(A)=\left\{ u\in \bigcap_{p\ge 1} W^{2,p}(\Omega); u,\Delta u\in C(\overline{\Omega}),\ u\lfloor_{\partial\Omega}=0\right\}.
\end{aligned}
\end{equation}
where $\Omega$ is bounded and its boundary is sufficiently smooth. It is known that the interpolation property
\begin{equation}
(X,D(A))_{\theta,\infty}=\{ u\in C^{2\theta}(\overline{\Omega});u\lfloor_{\partial\Omega} =0 \}
\end{equation}
is satisfied \cite{Lunardi2}.

\begin{theo}
	Let $T>0$ be arbitrary and $f\in {\cal F}^{\theta}([0,T]; C(\overline{\Omega}))$ for some $0<\beta <r\le 1$ and $0<\max(\beta, 1-\alpha)<r$. Then,
	\begin{equation}
	\left\{
	\begin{aligned}
	D_t^{\alpha}(u(\cdot,x)-u_0(x))(t)-\Delta u(t,x) - e^{-1/u}\chi_{\{u(t,x)>0\}}&=f(t,x),\ &&(t,x)\in(0,T]\times\Omega,\\
	u(t,x)&=0\ &&(t,x)\in [0,T]\times\partial\Omega,\\
	u(t,0)&=u_0(x)\ &&x\in \Omega
	\end{aligned}
	\right.
	\end{equation}
	has the unique solution 
	\[
	u\in C([0,T];C(\overline{\Omega}))\cap C((0,T];D(A))\cap {\cal F}^{r,\beta}((0,T];X)
	\]
	If $f$ is differentiable with respect to $t$ and its derivative is of class ${\cal F}^{r',\beta'}((0,T];C(\overline{\Omega})$ for $0<\beta'<r'\le 1$, $Au_0+B(u_0)-f(0)\in \overline{D(A)}$, then $u$ is time differentiable and $u'\in {\cal F}^{r',\beta'}((0,T];X)$.
\end{theo}

Proposition \ref{inclusion of domain power of op} enables us to estimate the $D(A^{q})$ from below and above, that is, for all $\theta>q$,
\[
\{u\in C^{2\theta}(\overline{\Omega}); u\lfloor_{\partial\Omega}=0 \} \subset D(A^{q}) \subset \{u\in C^{2q}(\overline{\Omega}) ; u\lfloor_{\partial\Omega}=0 \}.
\]

Proposition \ref{max} provides further properties to (\ref{combustion}).

\begin{prop}
	Assume $f\ge 0,\ u_0\ge 0$ and $f, u_0$ be sufficiently regular. Then, $u\ge 0$.
\end{prop}
\begin{proof}
	The selection of Banach space (\ref{continuous space}) and Proposition\ \ref{space regularity} yield the twice space differentiability. If the solution $u$ attain its minimum at $(t_0,x_0)\in (0,T]\times \Omega$ and $u(t_0,x_0)<0$, then Proposition\ \ref{max} provides
	\[
	0>D_t^{\alpha}(u(\cdot,x_0)-u_0(x_0))-\Delta u(t_0,x_0)= f(t_0,x_0),
	\]
	it is nothing but a contradiction.
\end{proof}

\section{Acknowledgement}
I would like to thank my instructor Michiaki Onodera, associate professor of the Department of Mathematics, School of Science, Tokyo Institute of Technology, for his helpful comments on this study.

\appendix
\section{Power of operator}
We provide the overview of the fractional power of operator. See \cite{Lunardi1}, \cite{Lunardi2}, \cite{Yagi} for details. Let $A$ be sectorial, then we define its power for $\Re z>0$ by
\begin{equation}\label{def of Az}
A^{-z}:=\frac{1}{2\pi i}\int_{\Gamma} \lambda^{-z}(\lambda -A)^{-1}d\lambda,
\end{equation}
where we take the branch
\[
\lambda^{-z}=\exp[-z (\log|\lambda| +i\arg\lambda) ],
\]
that is, single-valued function on $\mathbb{C}\setminus(-\infty,0]$. Also, we take the integral path $\Gamma=\Gamma_{-}\cup \Gamma_{+}\cup\Gamma_0$, where
\[
\begin{aligned}
&\Gamma_{\pm}:\lambda =\rho e^{\pm i\omega},\ \delta\le \rho<\infty,\\
&\Gamma_0:\lambda =\delta e^{i\varphi},\ -\omega\le \varphi \le \omega,
\end{aligned}
\]
for $\delta <\|A^{-1}\|^{-1}$. The evaluation
\[
|\lambda^{-z}|= |\exp{-z(\log\rho \pm i\omega)}|=e^{\pm (\Im z)\omega}\rho^{-\Re z},\ \lambda\in \Gamma_{\pm}
\]
guarantees that the definition of $A^{-z}$ makes sense. When $z=n\in\mathbb{N}$, we transform the path into $|\lambda|=\delta$ so that we confirm
\[
\begin{aligned}
&\frac{1}{2\pi i}\int_{C_{\delta}} \lambda^{-n}(\lambda-A)^{-1}d\lambda\\
&=-\frac{1}{(n-1)!} \left[ \frac{d^{n-1}}{d\lambda^{n-1}} (\lambda-A)^{-1} \right]_{\lambda=0}\\
&=(-1)^n [(\lambda-A)^{-n}]_{\lambda=0}\\
&=A^{-n}.
\end{aligned}
\]
where we retake the path $C_{\delta}=\{\lambda\in \mathbb{C}; |\lambda|=\delta\}$. For all $\Re z,\ \Re z'>0$, we can show the semigroup porperty
\[
A^{-z}A^{-z'}=A^{-(z+z')}
\]
by the standard way. We define $A^n$ as the inverse of $A^{-n}$. Let $z\in \mathbb{R}$ is not an integer. If $A^{-z} u=0$, then for any $n>z$,
\[
A^{-n}u =A^{-(n-z)}A^{-z}u =0
\]
Therefore $u=0$. This means that $A^{-z}$ has inverse. So we define $A^z=(A^{-z})^{-1}$, $D(A^z)=R(A^{-z})$. We can show $R(A^{-z_2})\subset R(A^{-z_1})$, i.e., $D(A^{z_2})\subset D(A^{z_1})$ since $0<\Re z_1<\Re z_2$ implies $A^{-z_2}=A^{-z_1}A^{-(z_2-z_1)}$.

We list the important properties for the power of operator theory. We refer to  \cite{Yagi}, \cite{Lunardi1}, \cite{Lunardi2} for the detailed proof.

\begin{prop}\label{inclusion of domain power of op}
	Let $A$ be a sectrial operator on Banach space $X$. For $0<\theta<\theta'<1$, the following inclusion holds.
	\[
	(X,D(A))_{\theta',\infty}\subset(X,D(A))_{\theta,1}\subset D(A^{\theta})\subset (X,D(A))_{\theta,\infty}
	\]
\end{prop}

\begin{prop}
	Let $\omega<\pi /2$. Then, for any $0<t<\infty$ and $0<\theta<\infty$,
	\[
	\| A^{\theta} T(t)\|_{B(X)} \le \frac{M \Gamma(\theta)}{\cos \omega} \frac{1}{t^{\theta}}
	\]
	holds.
\end{prop}
\begin{proof}
	Let 
	\[
	E^{\theta}(t):=\frac{1}{2\pi i} \int_{\Gamma} \lambda^{\theta} e^{-t\lambda} (\lambda -A)^{-1} d\lambda
	\]
	where
	\[
	\begin{aligned}
	&\Gamma_{\pm}:\lambda=\rho e^{\pm i \omega},\ 0\le \rho<\infty\\
	&\Gamma_0:\lambda =\delta e^{i \varphi},\ -\omega\le \varphi\le \omega.
	\end{aligned}
	\] 
	We can easly check $A^{-\theta}E^{\theta}(t)=E^{\theta}(t)A^{-\theta}=T(t)$, i.e.,
	\[
	A^{\theta} T(t)=\frac{1}{2\pi i} \int_{\Gamma} \lambda^{\theta} e^{-t\lambda} (\lambda -A)^{-1} d\lambda,
	\]
	which yields
	\[
	\begin{aligned}
	\| A^{\theta}T(t)\|_{B(X)} &\le \frac{M}{\pi} \int_{\delta}^{\infty} \rho^{\theta-1}e^{-t\rho \cos \omega} d\rho+ \frac{M}{2\pi} \int_{-\omega}^{\omega} \delta^{\theta} e^{-t\delta \cos \varphi}d\varphi.
	\end{aligned}
	\]
	Passing $\delta\rightarrow 0$ completes the proof.
\end{proof}

\section{Special functions}

We summarize some properties of special functions such as the Gamma function denoted by
\[
\Gamma(z):=\int_0^{\infty}t^{z-1}e^{-t}dt\ \ \Re{z}>0,
\]
and Beta function denoted by
\[
B(x,y):=\int_{0}^{1} (1-t)^{x-1}t^{y-1}dt\ \ \Re{x},\ \Re{y}>0.
\]
These functions satisfy the following properties (see for instance \cite{Podlubny});
\[
\begin{aligned}
&\mbox{(i)}\ \Gamma(1+x)=x\Gamma(x)\ \forall x>0,\\
&\mbox{(ii)}\ B(x,y)=\frac{\Gamma(x)\Gamma(y)}{\Gamma(x+y)},\\
&\mbox{(iii)}\ \frac{1}{\Gamma(x)}=\frac{1}{2\pi i} \int_{C_H} e^{\lambda}\lambda^{-x}d\lambda\ \ \mbox{(Hankel's formula)}
\end{aligned}
\]
where $C_H$ is an Hankel contour;
\[
\begin{aligned}
&C_H=C_+\cup C_{-}\cup C_0,\\ 
&C_{+}: \Im \lambda = \epsilon,\ \ \Re \lambda: -\infty\rightarrow 0,\\
&C_{-}: \Im \lambda = -\epsilon,\ \ \Re \lambda: 0\rightarrow -\infty,\\
&C_0:\epsilon e^{i\varphi},\ \ -\pi/2\le \varphi\le \pi/2.
\end{aligned}
\]
for some $\epsilon >0$. We also introduce the Mittag-Leffler function for $0\le t<\infty$
\[
E_{\mu, \nu}(t):=\sum_{n=0}^{\infty} \frac{t^{n\nu}}{\Gamma(\mu+n\nu)}
\]
which is evaluated as follows.

\begin{lem}(see \cite{Yagi}) We estimate
	\[
	E_{\mu,\nu}(t)\le \frac{2}{\Gamma_0 \nu} (1+t)^{2-\mu}e^{t+1}
	\]
	where $\Gamma_0= \min_{0<\xi<\infty} \Gamma(\xi)$.
\end{lem}

Suppose that $\varphi(t,s)$ is real valued function defined for $0\le s<t\le T$ satisfying 
\[
0\le \varphi(t,s)\le C(t-s)^{\epsilon-1}
\]
for some $\epsilon>0$ and $C>0$. We introduce the following integral inequality.

\begin{prop}\label{theo, volterra}
	(see \cite{Yagi}) Let $a,\ b,\ \mu,\ \nu>0$ be constant. If $\varphi$ satisfies
	\[
	\varphi(t,s)\le a (t-s)^{\mu-1}+b \int_s^t (t-\tau)^{\nu-1} \varphi(\tau,s)d\tau,\ \ 0\le s<t\le T,
	\]
	then, the following evaluation holds,
	\[
	\varphi(t,s) \le a\Gamma(\mu)(t-s)^{\mu-1} E_{\mu, \nu}\left[ \left( b\Gamma(\nu) \right)^{1/\nu} (t-s) \right],\ \ 0\le s<t\le T.
	\]
\end{prop}
Suppose that $w(t)$ is real valued function for $0<t\le T$ satisfying
\[
0\le w(t) \le C t^{\epsilon-1}
\]
for some $\epsilon,\ C>0$. Replacement of $\varphi(t,s)=w(t-s)$ yields the next corollary.

\begin{cor}\label{cor volterra}
	We make the same assumptions as in Proposition\ \ref{theo, volterra}. If $w$ meets
	\[
	w(t)\le a t^{\mu-1} +b\int_{0}^{t}(t-\tau)^{\nu-1}w(\tau)d\tau,
	\]
	then, $w$ satisfies
	\[
	w(t)\le a\Gamma(\mu)t^{\mu-1} E_{\mu, \nu}\left[ \left( b\Gamma(\nu) \right)^{1/\nu} t\right].
	\]
	Especially,
	\[
	w(t)\le \frac{2a\Gamma(\mu) }{\Gamma_0 \mu} t^{\mu-1} (1+b't)^{2-\mu} e^{b't+1}
	\]
	where $b'=(b\Gamma(\nu))^{1/\nu}$.
\end{cor}

\end{document}